\documentclass[12pt]{amsart}
\usepackage{amscd,amssymb,latexsym}
\usepackage{fullpage}
\newcommand{\Mdef}[2]{\newcommand{#1}{\relax \ifmmode #2 \else $#2$\fi}}

\newcommand{\imagesofgenerators}{\unitad\tensor \cG}

\newcommand{\modules}{\mbox{-modules}}
\newcommand{\cCad}{\unitad\modules}
\newcommand{\bd}{\mathbf{d}}

\usepackage{pigpen}
\usepackage{bbm}

\newcommand{\cok}{\mathrm{cok}}

\newcommand{\supp}{\mathrm{supp}}
\newcommand{\cosupp}{\mathrm{cosupp}}

\newcommand{\thick}{\mathrm{Thick}}
\newcommand{\loc}{\mathrm{Loc}}

\newcommand{\tensor}{\otimes}

\newcommand{\map}{\mathrm{Map}}
\newcommand{\Hom}{\mathrm{Hom}}

\Mdef{\bhom}{\mathbf{\hat{H}om}}

\Mdef{\Mod}{\mathrm{mod}}

\newcommand{\st}{\; | \;}

\newtheorem{thm}{Theorem}[section]
\newtheorem{lemma}[thm]{Lemma}
\newtheorem{prop}[thm]{Proposition}
\newtheorem{cor}[thm]{Corollary}

\theoremstyle{definition}

\newtheorem{defn}[thm]{Definition}
\newtheorem{defn-prop}[thm]{Definition--Proposition}

\newtheorem{example}[thm]{Example}

\newtheorem{remark}[thm]{Remark}

\newcommand{\qqed}{\qed \\[1ex]}
\renewenvironment{proof}[1][\hspace*{-.8ex}]{\noindent {\bf Proof #1:\;}}{\qqed}

\Mdef{\PH} {\Phi^H}
\Mdef{\PK} {\Phi^K}
\Mdef{\PL} {\Phi^L}
\Mdef{\PT} {\Phi^{\T}}

\Mdef{\ef}{E{\cF}_+}
\Mdef{\etf}{\widetilde{E}{\cF}}
\Mdef{\eg}{E{G}_+}
\Mdef{\etg}{\tilde{E}{G}}

\Mdef{\infl}{\mathrm{inf}}
\Mdef{\defl}{\mathrm{def}}
\Mdef{\res}{\mathrm{res}}
\Mdef{\ind}{\mathrm{ind}}
\Mdef{\coind}{\mathrm{coind}}

\Mdef{\univ}{\mathcal{U}}

\Mdef{\Fp}{\mathbb{F}_p}
\Mdef{\Zpinfty}{\Z /p^{\infty}}
\Mdef{\Zpadic}{\Z_p^{\wedge}}
\newcommand{\Zp}{\Z_{p}^{\wedge}}

\newcommand{\bi}{\begin{itemize}}
\newcommand{\be}{\begin{enumerate}}
\newcommand{\bc}{\begin{center}}
\newcommand{\ei}{\end{itemize}}
\newcommand{\ee}{\end{enumerate}}
\newcommand{\ec}{\end{center}}
\newcommand{\ed}{\end{description}}

\newcommand{\trichotomy}[3]{\left\{ \begin{array}{ll}#1\\#2\\#3 \end{array}\right.}

\newcommand{\lra}{\longrightarrow}

\newcommand{\spec}{\mathrm{Spec}}

\newcommand{\spcc}{\mathrm{Spc}^{\omega}}

\Mdef{\we}{\mathbf{we}}
\Mdef{\fib}{\mathbf{fib}}
\Mdef{\cof}{\mathbf{cof}}
\Mdef{\BI}{\mathcal{BI}}

\newcommand{\fibre}{\mathrm{fibre}}

\newcommand{\ilim}{\mathop{ \mathop{\mathrm{lim}} \limits_\leftarrow} \nolimits}
\newcommand{\colim}{\mathop{  \mathop{\mathrm {lim}} \limits_\rightarrow} \nolimits}

\Mdef{\B}{\mathbb{B}}
\Mdef{\C}{\mathbb{C}}
\Mdef{\D}{\mathbb{D}}
\Mdef{\E}{\mathbb{E}}
\Mdef{\T}{\mathbb{T}}
\Mdef{\F}{\mathbb{F}}
\Mdef{\G}{\mathbb{G}}
\Mdef{\I}{\mathbb{I}}
\Mdef{\N}{\mathbb{N}}
\Mdef{\Q}{\mathbb{Q}}
\Mdef{\R}{\mathbb{R}}
\Mdef{\bbS}{\mathbb{S}}
\Mdef{\Z}{\mathbb{Z}}

\Mdef{\bA}{\mathbb{A}}
\Mdef{\bB}{\mathbb{B}}
\Mdef{\bC}{\mathbb{C}}
\Mdef{\bD}{\mathbb{D}}
\Mdef{\bE}{\mathbb{E}}
\Mdef{\bF}{\mathbb{F}}
\Mdef{\bG}{\mathbb{G}}
\Mdef{\bH}{\mathbb{H}}
\Mdef{\bI}{\mathbb{I}}
\Mdef{\bJ}{\mathbb{J}}
\Mdef{\bK}{\mathbb{K}}
\Mdef{\bL}{\mathbb{L}}
\Mdef{\bM}{\mathbb{M}}
\Mdef{\bN}{\mathbb{N}}
\Mdef{\bO}{\mathbb{O}}
\Mdef{\bP}{\mathbb{P}}
\Mdef{\bQ}{\mathbb{Q}}
\Mdef{\bR}{\mathbb{R}}
\Mdef{\bS}{\mathbb{S}}
\Mdef{\bT}{\mathbb{T}}
\Mdef{\bU}{\mathbb{U}}
\Mdef{\bV}{\mathbb{V}}
\Mdef{\bW}{\mathbb{W}}
\Mdef{\bX}{\mathbb{X}}
\Mdef{\bY}{\mathbb{Y}}
\Mdef{\bZ}{\mathbb{Z}}

\Mdef{\cA}{\mathcal{A}}
\Mdef{\cB}{\mathcal{B}}
\Mdef{\cC}{\mathcal{C}}
\Mdef{\mcD}{\mathcal{D}} 
\Mdef{\cE}{\mathcal{E}}
\Mdef{\cF}{\mathcal{F}}
\Mdef{\cG}{\mathcal{G}}
\Mdef{\mcH}{\mathcal{H}} 
\Mdef{\cI}{\mathcal{I}}
\Mdef{\cJ}{\mathcal{J}}
\Mdef{\cK}{\mathcal{K}}
\Mdef{\mcL}{\mathcal{L}}

\Mdef{\cM}{\mathcal{M}}
\Mdef{\cN}{\mathcal{N}}
\Mdef{\cO}{\mathcal{O}}
\Mdef{\cP}{\mathcal{P}}
\Mdef{\cQ}{\mathcal{Q}}
\Mdef{\mcR}{\mathcal{R}}
\Mdef{\cS}{\mathcal{S}}
\Mdef{\cT}{\mathcal{T}}
\Mdef{\cU}{\mathcal{U}}
\Mdef{\cV}{\mathcal{V}}
\Mdef{\cW}{\mathcal{W}}
\Mdef{\cX}{\mathcal{X}}
\Mdef{\cY}{\mathcal{Y}}
\Mdef{\cZ}{\mathcal{Z}}

\Mdef{\ca}{\mathcal{a}}
\Mdef{\ct}{\mathcal{t}}

\Mdef{\At}{\tilde{A}}
\Mdef{\Bt}{\tilde{B}}
\Mdef{\Ct}{\tilde{C}}
\Mdef{\Et}{\tilde{E}}
\Mdef{\Ht}{\tilde{H}}
\Mdef{\Kt}{\tilde{K}}
\Mdef{\Lt}{\tilde{L}}
\Mdef{\Mt}{\tilde{M}}
\Mdef{\Nt}{\tilde{N}}
\Mdef{\Pt}{\tilde{P}}

\Mdef{\tA}{\tilde{A}}
\Mdef{\tB}{\tilde{B}}
\Mdef{\tC}{\tilde{C}}
\Mdef{\tE}{\tilde{E}}
\Mdef{\tH}{\tilde{H}}
\Mdef{\tK}{\tilde{K}}
\Mdef{\tL}{\tilde{L}}
\Mdef{\tM}{\tilde{M}}
\Mdef{\tN}{\tilde{N}}
\Mdef{\tP}{\tilde{P}}

\Mdef{\ft}{\tilde{f}}
\Mdef{\xt}{\tilde{x}}
\Mdef{\yt}{\tilde{y}}

\Mdef{\Ab}{\overline{A}}
\Mdef{\Bb}{\overline{B}}
\Mdef{\Cb}{\overline{C}}
\Mdef{\Db}{\overline{D}}
\Mdef{\Eb}{\overline{E}}
\Mdef{\Fb}{\overline{F}}
\Mdef{\Gb}{\overline{G}}
\Mdef{\Hb}{\overline{H}}
\Mdef{\Ib}{\overline{I}}
\Mdef{\Jb}{\overline{J}}
\Mdef{\Kb}{\overline{K}}
\Mdef{\Lb}{\overline{L}}
\Mdef{\Mb}{\overline{M}}
\Mdef{\Nb}{\overline{N}}
\Mdef{\Ob}{\overline{O}}
\Mdef{\Pb}{\overline{P}}
\Mdef{\Qb}{\overline{Q}}
\Mdef{\Rb}{\overline{R}}
\Mdef{\Sb}{\overline{S}}
\Mdef{\Tb}{\overline{T}}
\Mdef{\Ub}{\overline{U}}
\Mdef{\Vb}{\overline{V}}
\Mdef{\Wb}{\overline{W}}
\Mdef{\Xb}{\overline{X}}
\Mdef{\Yb}{\overline{Y}}
\Mdef{\Zb}{\overline{Z}}

\Mdef{\db}{\overline{d}}
\Mdef{\hb}{\overline{h}}
\Mdef{\qb}{\overline{q}}
\Mdef{\rb}{\overline{r}}
\Mdef{\tb}{\overline{t}}
\Mdef{\ub}{\overline{u}}
\Mdef{\vb}{\overline{v}}

\Mdef{\hc}{\hat{c}}
\Mdef{\he}{\hat{e}}
\Mdef{\hf}{\hat{f}}
\Mdef{\hA}{\hat{A}}
\Mdef{\hH}{\hat{H}}
\Mdef{\hJ}{\hat{J}}
\Mdef{\hM}{\hat{M}}
\Mdef{\hP}{\hat{P}}
\Mdef{\hQ}{\hat{Q}}

\Mdef{\thetab}{\overline{\theta}}
\Mdef{\phib}{\overline{\phi}}

\Mdef{\uA}{\underline{A}}
\Mdef{\uB}{\underline{B}}
\Mdef{\uC}{\underline{C}}
\Mdef{\uD}{\underline{D}}

\Mdef{\bolda}{\mathbf{a}}
\Mdef{\boldb}{\mathbf{b}}

\Mdef{\fm}{\frak{m}}
\Mdef{\fp}{\frak{p}}

\newcommand{\fX}{\mathfrak{X}}

\Mdef{\eps}{\epsilon}

\newcommand{\cell}{\mathrm{Cell}}

\input{xypic}

\newcommand{\fq}{\mathfrak{q}} 
\newcommand{\unit} {\mathbbm{1}}
\newcommand{\bbI}{\mathbb{I}}
\newcommand{\MCL}{MC_L}
\newcommand{\MCR}{MC_R}
\newcommand{\boldd}{\mathbf{d}}
\newcommand{\unitad}{\unit_{ad}}
\newcommand{\unitBP}{\unit_{BP}}
\newcommand{\bbD}{\mathbb{D}}

\newcommand{\sfD}{\mathsf{D}}

\newcommand{\cLim}{\mcL im}
\newcommand{\spcLim}{\mathrm{Skel}}
\newcommand{\bfD}{\mathbf{D}}
\newcommand{\thickt}{\mathrm{Thick}^{\otimes}}
\newcommand{\loct}{\mathrm{Loc}^{\otimes}}
\newcommand{\cCb}{\overline{\cC}}

\newcommand{\gen}{\mathfrak{g}}

\usepackage{enumitem}
\usepackage{bbm}
\usepackage{url}

\begin{document}
\title{Adelic models of tensor-triangulated categories}

\author[Balchin]{Scott Balchin}
\address[Balchin]{Warwick Mathematics Institute, Zeeman
  Building, Coventry, CV4 7AL, UK}
\author[Greenlees]{J.P.C. Greenlees}
\address[Greenlees]{Warwick Mathematics Institute, Zeeman
  Building, Coventry, CV4 7AL, UK}
 \date{\today}
\begin{abstract}
We show that a well behaved Noetherian, finite dimensional, 
stable,  monoidal model category has a model built from
categories of modules over completed rings in an adelic fashion. 
Special cases include abelian groups (the Hasse square), chromatic
homotopy theory (a module theoretic chromatic fracture square), and 
rational torus-equivariant homotopy theory (first step to the model of \cite{tnqcore}).  
\end{abstract}

\thanks{The authors are grateful to EPSRC for support from EP/P031080/1. 
The authors would also like to thank the Isaac Newton Institute for Mathematical Sciences, Cambridge, for support and hospitality during the programme 
\emph{Homotopy Harnessing Higher Structures}
where the first version of this paper was completed. This work was supported by EPSRC grant number EP/K032208/1.
 }

\maketitle

\tableofcontents

\section{Introduction}

This paper is concerned with models for well behaved tensor 
triangulated categories $\cCb$. Two motivating examples are the derived 
category $\sfD (R)$ of a commutative Noetherian ring $R$, and the category of 
rational torus-equivariant cohomology theories. Our results have new 
implications even  for the classical example $\sfD (R)$.

\subsection{Models of tensor-triangulated categories}
 This introductory section will be informal, and during the rest of this paper
 it will be made precise.  
To give us an appropriate context, we assume that our tensor
triangulated category  $\cCb$ is the homotopy category of a stable,
symmetric monoidal model category $\cC$.

There are a number of approaches to providing models of tensor-triangulated
categories $\cCb$.  The best known approach is Morita theory: if $\cCb$ has a small generator $k$ then $\cC$ is equivalent to the category of right modules over the endomorphism ring $\cE := \mathrm{End}(k)$, in a context where this makes sense \cite{SchwedeShipley}. This is an extremely powerful technique, and very valuable for many purposes, but it has the disadvantage that the ring $\cE$ is typically non-commutative and usually its coefficient ring $\pi_*(\cE)$ has infinite homological dimension. For our purposes these are serious disadvantages. 

The second approach is typified by the case when $\cCb=\sfD (R)$ is
the derived category of a commutative Noetherian ring $R$. It is well
known that the abelian category of $R$-modules is equivalent to the
category of quasi-coherent sheaves of modules over the topological space $\spec (R)$, and the corresponding statement applies at the level of derived categories. This method is also a very powerful technique: in effect it is reconstructing $R$-modules from the categories of modules over the local rings $R_{\fp}$ for primes $\fp$.
The disadvantages for us are (i) the category of modules over the local ring $R_{\fp}$ is still quite complicated and (ii) the process for assembling stalks into a sheaf is intricate.  

Nonetheless, this approach can be extended to general
tensor-triangulated categories. We replace $\spec (R)$  by the Balmer
spectrum $\spcc (\cCb) $ of small  objects in $\cCb$ (see Subsection \ref{subsec:ttcat} for definitions). In the case of $R$-modules there is an order reversing bijection $\spec (R)\cong \spcc(\sfD (R))$, so this does extend the classical case of commutative algebra. 

Our work goes further, by using simpler building
blocks to construct the model of $\cCb$. In the commutative algebra case,
we  reconstruct the category of  $R$-modules from the categories of modules over the
localized completed rings $(R_{\fp}^{\wedge})_{\fp}$. In fact we will
show that one can construct such a model rather generally, for
categories $\cCb$ where the Balmer spectrum has the formal properties
of the Zariski spectrum of a finite dimensional  Noetherian commutative ring.  In
this case $(R_{\fp}^{\wedge})_{\fp}$ is a ring  in the underlying category $\cC$, which will have a coefficient ring 
$\pi_*((R_{\fp}^{\wedge})_{\fp})$. This type of model has significant
advantages. First, the pieces are genuinely simpler, and so is the
assembly process. Second, and more important, we may
often prove  that the localized completed rings 
$(R_{\fp}^{\wedge})_{\fp}$ are also commutative. In that case, under
smoothness assumptions, they may be shown to  be formal in the sense that the
rings $(R_{\fp}^{\wedge})_{\fp}$  in $\cC$
are determined by their coefficient rings. The process of assembling
the model from these pieces is then also quite rigid. In this happy
situation, the category $\cC$ is determined up to Quillen equivalence
by the homotopy of the rings $(R_{\fp}^{\wedge})_{\fp}$ and and
therefore equivalent to one which is algebraic in nature. This happens
for rational torus-equivariant spectra \cite{tnqcore}. 

Tensor triangulated categories, localizations, completions and support
have been studied in various contexts including algebraic geometry, chromatic
and equivariant homtopy theory, commutative algebra and representation
theory. The terminology and standard references are not always the
same in the different traditions, and different facts are viewed as
being well-known, so we have tried to be thorough in documenting our
sources and apologize in advance if we have failed to give proper
credit in some cases.  In homotopy theory, related reconstruction methods  are used  in the recent
preprint~\cite{stratified} using categorical localizations of the
entire catgory, rather than the categories of
modules over rings used here. 

\subsection{Abelian groups}
To bring the discussion down to earth,  consider the category of abelian groups. The road to our 
model proceeds as follows. The first step is to recognize abelian groups as $\Z$-modules. Since these are modules in the original category, this looks perverse, but it provides a better starting point. 

We then attempt to analyse the category by analysing the unit object $\Z$, and for this purpose we consider the Hasse square
$$\xymatrix{
\Z\ar[d]\ar[r]&\Q\ar[d]\\
\prod_p\Zp\ar[r]&\Q \tensor \prod_p\Zp \rlap{ .}
}$$
The key to progress is that this is a pullback square. Classically, this enables number theoretic questions to be reduced to 
statements over $\Q$ and over $\Z_p^{\wedge}$. We are interested in the {\em derived} category of abelian groups, so we will use the fact that  the square is also a {\em homotopy} pullback. This lets us reconstruct the derived category of abelian groups from derived categories of $\Q$ and 
$\Zp$. The point of this is that modules over $\Q$ and $\Zp$ are much
simpler than modules over $\Z$, or even over the local rings
$\Z_{(p)}$. Similarly, if we start with the derived category of a
finite dimensional commutative Noetherian ring then we may reconstruct the derived category of $R$-modules from the derived categories of the localized completed rings $(R_{\fp}^{\wedge})_{\fp}$. 

\subsection{Outline of paper}

We work with tensor-triangulated categories $\cCb$ which we think
 of as analogues of categories of modules over a commutative ring. We will impose restrictions on these categories so that they
have the same formal properties as categories of modules of finite dimensional
Noetherian rings. The principal input  from the category $\cCb$ is its 
 Balmer spectrum, which is the categorical analogue of the Zariski spectrum of
a commutative ring. We recall the necessary definitions  in Section~\ref{sec:ttcat2}.

In particular, we can discuss localization and completion of objects
at some prime of the Balmer spectrum. We recall the classical theory
of localization and completion in derived commutative algebra in
Section~\ref{sec:fil} along with the associated notions of support and
cosupport. 

We wish to apply these ideas more generally, in the context where the
tensor-triangulated categories are homotopy categories of
well-structured Quillen model categories as discussed in Section~\ref{sec:modelcat}.
In the language of model categories,  localization and completion are
two examples of  Bousfield localizations: Section~\ref{sec:localize}
discusses localization, and  Section~\ref{sec:compp} discusses completion.

The fundamental result is the Adelic Approximation Theorem~\ref{aat}
which shows the monoidal unit of $\cCb$ is a homotopy limit in a
cubical diagram of
products of localized completed rings, directly analogous to the
Hasse square. 

From the Adelic Approximation Theorem we deduce our main
result (Theorem~\ref{mame}) giving the Adelic Model for $\cC$ in terms
of modules over the adelic rings. The further step of giving a model in
terms of modules over individual localized completed rings 
will be dealt with in a sequel \cite{adelics}.

\renewcommand{\arraystretch}{1.25}
\begin{table}[h!]
\begin{tabular}{|c|c|}
\hline
\textbf{Commutative algebra} & \textbf{Adelic models}                               \\ \hline
$\textbf{Ch}(R\textbf{-mod})$                          & $\cC$                                        \\ \hline
$\sfD (R)$                          & $\cCb$                                        \\ \hline
$R$                          & $\mathbbm{1}$                                        \\ \hline
$R_\mathfrak{p}$             & $L_\mathfrak{p}\mathbbm{1}$                          \\ \hline
$R^\wedge_\mathfrak{p}$      & $\Lambda_\mathfrak{p}\mathbbm{1}$ \\ \hline
$\spec(R)$                   & $\spcc(\cCb)$                                        \\ \hline
Hasse cube                   & Adelic approximation cube                                       \\ \hline
\end{tabular}
\caption{Commutative algebra and general adelic constructions}
\end{table}
\renewcommand{\arraystretch}{1.0}

\subsection{Notation}

We write $\cC$ for the model category we are studying, and $\cCb$ for
its homotopy category. The monoidal unit is denoted $\unit$. We shall use $\fp$ and $\fq$ when we discuss
Balmer primes, we generally assume that containment corresponds to
alphabetical order so that $\fp\supseteq
\fq$. We will use this Balmer ordering in all cases, even for
$\sfD(R)$ where it corresponds to the reverse of the commutative algebra
ordering.  The functor $L_\fp$ will be the Bousfield localization at $\fp$
and $\Lambda_\fp$ the Bousfield completion at $\fp$. We write
$\spcc(\cCb)$ for the Balmer spectrum of the triangulated
category of small objects $\cCb^\omega$.

When it comes to finiteness properties, all conventions contradict
some respected source: for us `compact', `finite' and `perfect'
apply to the model category whilst `small' and `rigid' refer to the homotopy
category. Precise definitions are given in Sections \ref{sec:ttcat2}
and \ref{sec:modelcat}.

\subsection*{Acknowledgements} 

The authors would like to thank to Tobias Barthel for many interesting
conversations and comments on this work, and the referee for supplying
several useful references and observing some of our original hypotheses were
unnecessary.

\section{Tensor-triangulated categories}\label{sec:ttcat2}
In this section we introduce the basic language at the homotopy
category level. This is in the context of tensor-triangulated
categories where the triangulated structure is augmented by a
symmetric monoidal tensor product exact in each variable: 
the survey by Stevenson \cite{StevensonTT} lays out the context.

\subsection{Basic terminology}\label{subsec:ttcat}
We recall some standard terminology from the study of 
tensor-triangulated categories and  the basic
definitions from \cite{BalmerSpc}. 

If $\cCb$ is a tensor-triangulated category, an object $T$ is called {\em 
small}  if  for any set of objects $Y_i$, the natural map 
$$\bigoplus_i [T,Y_i]\stackrel{\cong}\lra [T, \bigvee_i Y_i]$$
is an isomorphism (where $[A,B]$ denotes the $\cCb$-morphisms from $A$
to $B$). We warn that these are sometimes called `compact' or `finite'
but we are reserving those terms for model category level notions. 
We write $\cCb^\omega$ for the triangulated subcategory of
small objects.

We say that  a full subcategory $A$ of $\cCb$
is {\em thick} if it is closed under completing triangles and taking retracts. 
  It is
{\em localizing } if it is closed under completing triangles and
taking arbitrary
coproducts (it is then automatically closed under retracts as
well).  We say that $A$ is an {\em ideal} if it is closed under
triangles and tensoring with an arbitrary element. 

For a general subcategory $B$ we write $\thick(B)$ for the
thick subcategory generated by $B$  and $\thickt(B)$ for the
thick tensor ideal generated by $B$. The latter depends on the ambient
category, and we will only write $\thickt(B)$ in the category $\cCb^\omega$ of
small objects (so $B$ is small, and only tensor products with
small objects are permitted).  We write $\loc(B)$ for the 
localizing subcategory generated by $B$, and
$\loct(B)$ for the  localizing tensor ideal generated by $B$; because
an infinite coproduct of small objects will usually not be small, 
 this only makes sense for the full category $\cCb$ and tensor products with 
arbitrary objects of $\cCb$ are permitted. 
 
We will generally be interested in thick and localizing tensor
ideals, because without closure under tensor products  the structure
is hard to understand.

A triangulated category $\cCb$ is \emph{small-generated} if there is a set $\cG$ of small objects in $\cCb$ such that an object $X \in \cCb$ is zero if and only if 
$$[G, \Sigma^iX]=0 \; \text{ for every } \; G \in \cG \; \text{ and } \; i \in \mathbb{Z}.$$

In a small-generated tensor-triangulated category,  the tensor
product has a right adjoint, and therefore  it is in fact a closed
monoidal category, in that we have an internal hom right adjoint to
$\tensor$, which we shall denote $\underline{\mathrm{hom}}(-,-) \in
\cCb$. For every object $X \in \cCb$, we define its \emph{dual} to be
the object $DX := \underline{\mathrm{hom}}(X,\unit)$.  We shall say
that $X$ is \emph{rigid} (or \emph{strongly dualizable}) if the
natural map $X \tensor DX \xrightarrow{} \underline{\hom}(X,X)$ is an equivalence. 

The usual context for working with tensor-triangulated categories is as follows. 

\begin{defn}
A \emph{rigidly small-generated tensor-triangulated category}
$\cCb$ is a small-generated tensor-triangulated category such that
the small objects $X \in \cCb^\omega$ coincide with the
rigid objects (in particular, the tensor product of small
objects is small). 
\end{defn}

\subsection{The Balmer spectrum and support of small objects}
We may now introduce the organizational principle on which the construction is based.

\begin{defn}
A {\em prime ideal} in a tensor-triangulated category is a 
proper thick tensor ideal $\fp$ with the
property that $a\tensor b\in \fp$ implies that $a$ or $b$ is in $\fp$. 

The {\em Balmer spectrum} of a tensor-triangulated category $\cCb $ is
the set of prime tensor ideals of the triangulated category of small objects: 
$$\spcc (\cCb)=\{ \fp \subseteq \cCb^\omega \st \fp \mbox{  is prime } \}.$$ 

We may use this to define the {\em support} of a small object $X$:
$$\supp (X)=\{ \fp \in \spcc(\cCb)\st X\not\in \fp\}. $$
This in turn lets us define the  Zariski topology on $\spcc (\cCb)$ as 
generated by the closed sets $\supp (X)$ as  $X$ runs through small
objects of $\cCb$. We note that the set $\spcc(\cCb)$ equipped with
the Zariski topology is always a spectral space in the sense of Hochster~\cite{BuanSpectral}.
\end{defn}

\begin{example}
The motivating example is that if $\cCb =\sfD (R)$ is the derived category of a
commutative Noetherian ring $R$ then there is a natural homeomorphism
$$\spec (R)\stackrel{\cong}\lra \spcc (\sfD (R))$$
where the classical {\em algebraic} prime $\fp_a$ corresponds to the
{\em Balmer}
prime $\fp_b=\{ M \st M_{\fp_a}\simeq 0\}$. To avoid disorientation it
is essential to emphasize that   this is order-reversing, so that maximal algebraic primes
correspond to minimal Balmer primes; either way these are the closed
points. Even in this classical case we will use the Balmer order on
primes. 
\end{example}

We indicate the (specialization) closure operation by a bar:
$$\overline{\{\fp\}}=\Lambda (\fp)=\{ \fq \st \fq \subseteq \fp\}. $$

We say that a  prime $\fp$ is \emph{visible} if its closure is of the
form $\supp(M)$  a small object $M$ of $\cCb$ (the term is also used
for a weaker notion as in \cite{MR3591157}).

The results of this paper require all primes to be visible; if it fails, there is an additional layer of complication. 
The point is that if every prime is visible then the topology on $\spcc(\cCb)$ is generated by the 
closures of points, and hence determined by the poset structure of $\spcc(\cCb)$.  To explain this, recall that a topological space is said to be \emph{Noetherian} if its open sets satisfy the ascending chain condition. 

\begin{lemma}[{\cite[Corollary 2.17]{BalmerSpc}}]\label{lem:noeth}
The topological space $\spcc(\cCb)$ is Noetherian if and only if any closed subset of $\spcc(\cCb)$ is the support of an object in $\cCb^\omega$.\qqed
\end{lemma}

The space $\spcc(\cCb)$ is Noetherian if and only if all primes are
visible~\cite[Corollary 7.14]{BalGioTelescope}. 

\subsection{Classification of thick tensor ideals}
The purpose of the Balmer spectrum is to give a systematic approach to
the classification of thick tensor ideals. Balmer shows
\cite{BalmerSpc} that there is a bijection between the collection of
Thomason subsets $Y$ (i.e., $Y$ is a union of closed sets whose
complements are compact) of $\spcc(\cCb)$ and the radical thick
 tensor ideals of $\cCb^\omega$. In the case when all small objects are
 rigid, thick tensor ideals are automatically radical. We shall make two uses of this fact.
 
\begin{lemma}
\label{lem:Kfp}\leavevmode
\begin{itemize}
\item[(i)]  If $K_{\fp}, K'_{\fp}$ are two small objects with support $\Lambda
(\fp)=\overline{\{ \fp \}}$ then they generate the same thick tensor
ideal.
\item[(ii)] If $K_i$ are small objects whose supports cover $\spcc(\cCb)$
and $\cCb$ is small-generated then the objects $K_i$  generate
$\cCb$ as a localizing tensor ideal. \qqed
\end{itemize}
\end{lemma}

\subsection{Dimension and the partial order}

The primes in the Balmer spectrum are partially ordered by
inclusion. The minimal elements are the closed points, and every prime
contains a closed point (\cite[2.12]{BalmerFilt}).  We will filter
this by the {\em Krull dimension} \cite[3.2]{BalmerSpc} 
$$\dim(\fp)=\sup\{ m \st \exists \fp_0\subset \fp_2 \subset \cdots
\subset \fp_m=\fp\}. $$
We will henceforth simply refer to this as `dimension'.

We will restrict attention to tensor-triangulated
categories which are  rigidly small-generated with $\spcc(\cCb)$
finite dimensional and Noetherian.

\section{Localization and completion in the homotopy category}\label{sec:fil}
We begin by recalling the standard constructions of localization at a
prime, torsion at a prime and completion at a prime from commutative
algebra, since our analysis is based on counterparts of them. We also explain how to
decompose the category $\cCb$ corresponding to natural structures on 
$\spcc(\cCb)$, these are widely used in equivariant topology and
commutative algebra (see \cite{loccoh, GMcomp, DGtec, BalmerFilt, BalmerSpc,
  BIKcosupp, MR3591157}).

 In Sections \ref{sec:localize} and \ref{sec:compp} we will introduce corresponding constructions
for model categories and freely apply the language introduced here. 

\subsection{Localization, completion and local cohomology in
  commutative algebra}
\label{subsec:commalg}
Our construction is based on commutative algebra, so we need to explain the localization and completion from commutative 
algebra that we need in a way that makes it clear how to extend it. 

In this subsection we fix a commutative Noetherian ring $R$. We will continue to use
$\fp_a$ for the algebraic prime in the commutative ring $R$ (which is a subset of $R$) and $\fp_b$ for the corresponding Balmer prime (which is a collection of small objects of the derived category). 

 For localization at  $\fp_a$, we write $L_{\fp_a}M=M_{\fp_a}$.
In the classical world, this is obtained by inverting the multiplicative set $R\setminus \fp_a$. 
However in terms of the Balmer spectrum, for any object $M$ of the derived category $\sfD (R)$,   the localization map 
$$M \lra L_{\fp_a}M$$
is the initial map nullifying all small objects $T$ with $T_{\fp_a}\simeq 0$. In other words it is the nullification of the Balmer prime $\fp_b$.

Next, we define the $\fp_a$-power torsion and local cohomology. For an $R$-module
$M$, at the level of abelian categories the $\fp_a$-power torsion is defined by
$$\Gamma'_{\fp_a}M:= \{x \in M \st (\fp_a)^nx=0 \mbox{ for } n>>0\} $$
(the dash is to distinguish this from the derived version, which we will use constantly 
throughout the rest of this paper). 
The functor $\Gamma'_{\fp_a}$ is left exact and if  $R$ is Noetherian, the derived  functors are calculated by the stable Koszul complex. To define this, suppose  $\fp_a =(x_1, \ldots , x_n)$ 
and define 
$$\Gamma_{\fp_a}M=(R\lra R[1/x_1])\tensor_R\cdots \tensor_R (R\lra 
R[1/x_n])\tensor_RM. $$
Up to equivalence this is independent of generators, and the  cohomology of this complex is by definition the local cohomology; Grothendieck observed that it calculates the right derived functors of $\Gamma'_{\fp_a}$. 
From our point of view the important thing is that for any object $M$ of the derived category $\sfD (R)$, the map
$$\Gamma_{\fp_a}M \lra M$$ 
has the universal property of  $K_{\fp_a}$-cellularization where $K_{\fp_a}$ is the unstable
Koszul complex 
$$K_R(x_1, \ldots, x_n)=(R\stackrel{x_1}\lra R)\tensor_R\cdots \tensor_R (R\stackrel{x_n}\lra 
R). $$
This object $K_{\fp_a}$ depends on the generators, but the cellularization only depends on the fact that $K_{\fp_a}$ is a small object with 
support $\overline{\{\fp_b\}}$. 

These two functors  often occur together.  We note that the composite is smashing in the 
sense that  
$$\Gamma_{\fp_a}L_{\fp_a}M \simeq (\Gamma_{\fp_a}L_{\fp_a}R)\tensor_R M . $$

\begin{remark} We warn that \cite{BIKcosupp, StevensonTT} use $\Gamma_{\fp_a}$ for the composite 
$$L_{\fp_a}\Gamma_{\fp_a}M=\Gamma_{\fp_a}L_{\fp_a}M, $$
which we will never do. 
\end{remark}

For $\fp_a$-completion we start from the $\fp_a$-cellularization. We use the traditional notation 
 $\Lambda_{\fp_a}M$ for the derived $\fp_a$-completion functor,  
so that 
$$\Lambda_{\fp_a}M=\Hom_R(\Gamma_{\fp_a}R, M). $$
For a Noetherian ring $R$ and a module $M$,  the homotopy groups of
this are the left derived functors of $\fp_a$-adic completion \cite{GMcomp}. For a general object $M$ of the derived category $\sfD (R)$, 
the map 
$$M \lra \Lambda_{\fp_a}M$$
is the Bousfield localization with respect to $K_{\fp_a}$. 

Finally, we also write
$$V_{\fp_a}M=\Hom_R(L_{\fp_a}R, M).$$
The functors $\Lambda_{\fp_a}$ and $V_{\fp_a}$ often occur together, and 
we note that 
$$\Lambda_{\fp_a}V_{\fp_a}M\simeq \Hom_R(\Gamma_{\fp_a}L_{\fp_a}R, M).$$

\begin{remark}
We note that \cite{BIKcosupp} uses $\Lambda_{\fp_a}$  for the composite 
$$V_{\fp_a}\Lambda_{\fp_a}M=\Lambda_{\fp_a}V_{\fp_a}M, $$
which we will never do. 
\end{remark}

\subsection{Support and cosupport for arbitrary objects}
\label{subsec:suppcosupp}
We have defined support for small objects in terms of the primes, and we now extend this to general objects. 

\begin{defn}
\label{defn:suppcosupp} \cite{BalGioTelescope, BIKcosupp}
The support and cosupport of an $R$-module $M$ are defined by 
$$\supp (M)=\{ \fp\st \Gamma_{\fp}L_{\fp}R\tensor_R M\not \simeq
0\}. $$
$$\cosupp(M): =\{ \fp \st V_{\fp}\Lambda_{\fp}M\not \simeq 0\}
=\{ \fp \st \Hom_R(\Gamma_{\fp}R_{\fp}, M)\not \simeq 0\}. $$
\end{defn}

\begin{remark}
When $M$ is small,  the support is 
$$\{ \fp_a \st M_{\fp_a}\not \simeq 0\}=\{ \fp_a\st M\not \in \fp_b\}, $$
 but in general the support is a proper subset of $\{\fp_a \st
 M_{\fp_a}\not \simeq 0\}$. 
\end{remark}

The main fact that we shall use is that, for categories with a model, an object is trivial if it has
empty support (\cite{MR3181496})  and hence if it has empty cosupport.

\subsection{Filtrations by support}
The following filtrations are well known by various names in equivariant topology,
representation theory and algebraic geometry. We consider collections  $\cF$ of primes closed under 
specialization (`families') and collections  $\cG$ of primes closed under 
generization (`cofamilies'). If $\cF$ is a family, we write
$\tilde{\cF}$ for the complementary cofamily. 

In particular we consider the cones above and below a fixed prime  $\fq$: 
$$\Lambda (\fq )=\{ \fp\st 
\fp\leq \fq \}  \mbox{ and } V(\fq )=\{ \fp\st \fq \leq \fp \}   .$$
The first is a family (namely the closure of $\{ \fq\}$) and the 
second is a cofamily. 

\newcommand{\LcFt}{L_{\tilde{\cF}}}
Given a family $\cF$,  we may consider the set of Koszul objects for primes in $\cF$. Taking the cellularization with respect
to these small objects gives  $\Gamma_{\cF}X$ (so that
$\Gamma_{\fp}=\Gamma_{\Lambda (\fp)}$) and the nullification gives
$\LcFt X$ (so that $L_{\fp}=L_{V(\fp)}$). We then have a natural  cofibre sequence
$$\Gamma_{\cF}X \lra X\lra \LcFt X  $$
with 
$$\supp (\Gamma_{\cF}X)=\supp  (X) \cap \cF \mbox{  and }\supp
(\LcFt X)=\supp  (X) \setminus \cF. $$ 
All maps $\Gamma_{\cF}X \lra \LcFt X$ are trivial so this gives a
semi-orthogonal decomposition of the category.

A map $X\lra Y$ is an $\cF$-equivalence if $\Gamma_{\cF}X\lra
\Gamma_{\cF}Y$ is an equivalence or equivalently if $\supp (C)\cap
\cF=\emptyset$ 
where $C$ is the cofibre of $X\lra Y$.

If $\cF$ is the family of primes of dimension $\leq i$ and
$\tilde{\cF}$ is the complementary cofamily of primes of dimension
$\geq i+1$ we  write 
$$M_{\leq i}\lra M\lra M_{\geq i+1}$$
for the cellularization and nullification.  We say that a map $X\lra
Y$ is a $(\leq i)$-equivalence if it induces an equivalence $X_{\leq
  i}\lra Y_{\leq i}$, or equivalently if it is a equivalence when
tensored with any small object $K$ with $\supp (K)$ consisting of
primes of dimension $\leq i$.

\section{Model categories}\label{sec:modelcat}
To start with, we recall basic terminology and constructions from
model categories, and their relation to structures at the triangulated
category level.  The discussion of  diagrams and limits of model categories in
Subsection \ref{subsec:diagrammodcats} provides an essential framework for our adelic model.

\subsection{Compact, finite and small}
The terminology describing finiteness properties is in chaos, in the
sense that there is no single set of conventions applying
consistently to both model categories and homotopy
categories. We use conventions that are necessarily non-standard, but we hope
they are clear. The basic principle is that compact, finite and perfect are reserved for model 
category level concepts while small and rigid are exclusively for the derived category level. 

Most straightforward is the analogue of a finite cell complex. For a
set $I$ of morphisms, an
$I$-cell complex is an object constructed as a transfinite composition
of pushouts of elements of $I$ (see \cite[10.5.8]{Hirschhorn}), and it is a
finite $I$-cell complex if only finitely many steps are required. 

There are two counterparts of properties of compact topological spaces
that are important. First (colimit-compactness) that when mapping into
a sequential colimit they map
into some finite part, and second (cell-compactness)  that when mapping into a CW complex
they map into a finite part. 

\begin{defn}[Colimit-compactness]
Let $\cC$ be a  category, and $\mcD$ a collection of morphisms. 
\begin{enumerate}
\item If $\gamma$ is a cardinal, then an object $W$ in $\cC$ is
  \emph{$\gamma$-colimit-compact} over $\mcD$ if for all $\gamma$-filtered
  ordinals $\lambda$ and all diagrams $X: \lambda \lra \mcD$, we have
$$\colim_{\beta<\lambda} \cC (W, X(\beta))\stackrel{\cong}\lra \cC(W,
\colim_{\beta<\lambda}X(\beta)).$$
\item An object $W$ in $\cC$ is \emph{colimit-compact} over $\mcD$ if there is a cardinal $\gamma$ for which it is $\gamma$-compact. 
\end{enumerate}
Colimit-compactness is referred to as smallness in Hovey
\cite{HoveyMC} and Hirschhorn \cite{Hirschhorn}. 
\end{defn}

\begin{defn}[Cell-compactness]
Let $\cC$ be a cofibrantly generated model category with generating cofibrations $I$. 
\begin{enumerate}
\item If $\gamma$ is a cardinal, then an object $W$ in $\cC$ is \emph{$\gamma$-cell-compact} if it is $\gamma$-compact over $I$. That is, for every relative $I$-cell complex $f \colon X \to Y$, every map from $W$ to $Y$ factors through a subcomplex of size at most $\gamma$. 
\item An object $W$ in $\cC$ is \emph{cell-compact} if there is a cardinal $\gamma$ for which it is $\gamma$-cell-compact. 
\end{enumerate}
Cell-compactness is referred to as compactness in Hirschhorn \cite{Hirschhorn}. 
\end{defn}

\begin{defn}[{\cite[Definition 1.2.3.6]{MR2394633}}]
An object $X$ in $\cC$ is {\em perfect} if its image in $\cCb$ is
rigid in the sense of Subsection \ref{subsec:ttcat}. 
\end{defn}

\subsection{From model category to homotopy category}

We impose conditions on our model category $\cC$ that lead 
to the homotopy category being a well behaved rigidly small-generated tensor-triangulated category. 

Recall that a model category is said to be \emph{symmetric monoidal} if the underlying category is symmetric monoidal $(\otimes, \mathbbm{1})$ such that $\otimes \colon \mathcal{C} \otimes \mathcal{C} \to \mathcal{C}$ is a left Quillen bifunctor and such that for every cofibrant object $X$ and every cofibrant resolution $\emptyset \hookrightarrow Q \mathbbm{1} \to \mathbbm{1}$ of the tensor unit there is an induced weak equivalence
$$Q \mathbbm{1} \otimes X \to \mathbbm{1} \otimes X \xrightarrow{\simeq} X.$$

\begin{prop}
\label{prop:ssmmc}
{\cite[Theorem 4.3.2, \S 7.1, Theorem 6.6.4]{HoveyMC}}
Let $\cC$ be a stable and symmetric monoidal model category, then
$\cCb$  is a tensor-triangulated category with the induced tensor and
unit. 
\end{prop}

\begin{remark}
Note that Proposition \ref{prop:ssmmc} is not stated as such in \cite{HoveyMC}.  At the time, there was a conjectural assumption that we needed for this to hold, however, this conjecture (\cite[Conjecture 5.7.5]{HoveyMC}) has been proved to hold in generality in \cite{Cisinskicojecture}. 
\end{remark}

\begin{defn}[{\cite[Definition 12.1.1]{Hirschhorn}}]
A model category $\cC$ is said to be \emph{cellular} if it is cofibrantly generated with generating cofibrations $I$ and generating acyclic cofibrations $J$ such that
\begin{itemize}
\item All domain and codomain objects of elements of $I$ are
  cell-compact objects over $I$.
\item The domain objects of the elements of $J$ are colimit-compact objects over $I$.
\item The cofibrations are effective monomorphisms.  That is, a cofibration $f \colon A \to B$ is the equalizer of the pair of natural inclusions $B \rightrightarrows B \coprod_A B$.
\end{itemize} 
\end{defn}

We now strengthen a cellular model category to a \emph{compactly generated} category.

\begin{defn}
Let $\cC$ be a model category with a set of generating cofibrations $I$.  We say that $\cC$ is \emph{compactly generated} if
\begin{enumerate}
	\item The model category $\cC$ is cellular.
	\item There exists a set of generating cofibrations $I$ and
          generating acyclic cofibrations $J$ whose domains and
          codomains are cofibrant, $\omega$-cell-compact relative to $I$ and $\omega$-colimit-compact with respect to the whole category $\cC$.
	\item Filtered colimits commute with finite limits in $\cC$.
\end{enumerate}
\end{defn}

\begin{lemma}\label{generators}
Suppose that $\cC$ is a stable and compactly generated symmetric monoidal model category with generating cofibrations $I$, then the tensor-triangulated category $\cCb$ has a set of small generators given by the set $\cG$ of cofibres of maps in $I$.
\end{lemma}

\begin{proof}
The property of being compactly generated implies that the model category is \emph{finitely generated} in the sense of \cite[\S 7.4]{HoveyMC}.  We can then use \cite[Corollary 7.4.4]{HoveyMC} to obtain the result.
\end{proof}

\begin{thm}[{\cite[Corollary 1.2.3.8]{MR2394633}}]\label{cgmodels}
Suppose that $\cC$ is a stable and compactly generated symmetric monoidal model category with $\unit$ being $\omega$-cell-compact and cofibrant. We assume that the set $I$ of morphisms of the form 
$$S^n \tensor \cG \to \Delta^{n+1} \otimes \cG$$
is a set of generating cofibrations for $\cC$.  Here, $\cG$ is a set of $\omega$-cell-compact and $\omega$-colimit-compact cofibrant perfect objects.  Then the rigid objects and the small objects in $\cCb$ coincide, and these coincide with the retracts of finite $I$-cell complexes.
\end{thm}

\begin{remark}
In the above theorem we are using the fact that any model category $\mathcal{C}$ with functorial factorizations has the property that $\operatorname{Ho}(\mathcal{C})$ has the structure of a $\operatorname{Ho}(\textbf{sSet})$-module via the theory of framings~\cite[\S5]{HoveyMC}. This allows us to form the tensor products $S^n \otimes \mathcal{G}$ and $\Delta^{n+1} \otimes \mathcal{G}$.
\end{remark}

Note that by using \cite[Proposition 1.2.3.7]{MR2394633} it is possible to see that under the assumptions of the above theorem, the small objects form a tensor-triangulated subcategory of $\cCb$.  This then leads us to the following corollary.

\begin{cor}\label{cor:rcg}
Suppose that $\cC$ is as in Theorem \ref{cgmodels}, then $\cCb$ is a rigidly small-generated tensor-triangulated category with the compact generators being the objects of the set $\cG$.
\end{cor} 

\begin{defn}\label{rcgdef}
We say that a symmetric monoidal model category $\cC$ is \emph{rigidly-compactly generated} if
\begin{enumerate}
\item  It is stable, proper, and compactly generated.
\item The monoidal unit $\unit$ is $\omega$-cell-compact and cofibrant.
\item There is a generating set of cofibrations of the form $S^n
  \tensor \cG \to \Delta^{n+1} \otimes \cG$ with $\cG$  a set of
  $\omega$-cell-compact and $\omega$-small cofibrant  objects whose images
  in $\cCb$ are rigid. 
\end{enumerate}
\end{defn}

In light of Definition~\ref{rcgdef} and Corollary~\ref{cor:rcg} we conclude

\begin{cor}
\label{cor:fdcNoeth}
Let $\cC$ be a rigidly compactly generated symmetric monoidal model
category so that $\spcc(\cCb)$ is finite dimensional and Noetherian  then $\cCb$ is a Noetherian
tensor-triangulated category. 
\end{cor}


\begin{example}\label{runningex1}
If $R$ is a commutative Noetherian ring then the (unbounded) derived category $\cCb=\sfD (R)$
 is a rigidly small-generated tensor-triangulated
 category~\cite[Example 3.7]{MR3181496}. 
 It is the homotopy category of the rigidly compactly generated monoidal model category $\cC=\textbf{Ch}(R)$ equipped with the projective model structure.  The Balmer spectrum of $\sfD  
 (R)^\omega$ can be identified with $\operatorname{Spec}(R)$~\cite[\S  
 5]{BalmerSpc}.  The category $\textbf{Ch}(R)$ is an Noetherian model
 category, with the same dimension as $R$.
\end{example}

\begin{example}\label{runningex2}
If $G$ is a compact Lie group then the category of rational
$G$-equivariant cohomology theories is a rigidly small-generated
tensor-triangulated category.  It is the homotopy category of the rigidly compactly
generated monoidal model category of rational orthogonal $G$-spectra~\cite{MMorthog}.  

The Balmer spectrum can be identified with the space of
conjugacy classes of closed subgroups \cite{spcgq}. We have
$\fp_K\subseteq \fp_H$ if and only if $K$ is cotoral in $H$ (i.e., 
 conjugate to a subgroup $K'$ normal in $H$ with $H/K'$ a torus). 
The rank of the subgroup agrees with the Krull dimension on
$\spcc (\cCb)$. The spectrum is   Noetherian if $G$ is a torus,  but in  
 general the closures of points do not generate the topology.
\end{example}

\subsection{Diagram model categories}
\label{subsec:diagrammodcats}
We will several times need to consider generalized diagram 
categories, and we briefly recall the construction.

\begin{defn}
Let $\bfD$ be a small category, and $\cM$ a diagram of 
model categories indexed by $\bfD$.  That is, for each $s \in \bfD$,
we have a model category $\cM(s)$ and for each $a \colon s \to t$ in
$\bfD$, a left Quillen functor $a_\ast \colon \cM(s) \to \cM(t)$ (with
right adjoint $a^\ast$).  Then an \emph{$\cM$-diagram} $X$ specifies
for each object $s$ in $\bfD$ an object $X(s)$ of $\cM(s)$ and for
each morphism $a \colon s \to t$ in $\bfD$ a {\em base change map} 
$\widetilde{X}(a) \colon a_\ast X(s) \to X(t)$ pseudofunctorial for
composition. When $\cM$ is not mentioned, we refer to this as a
{\em generalized diagram}. 
\end{defn}

\begin{prop}[{\cite[Theorem 3.1, Proposition 3.3]{diagrams}}]\label{diamod}
Suppose given  a diagram of model 
categories $\cM$ indexed on $\bfD$.  

(i) If $\bfD$ is a direct category, 
there is a diagram-projective   model structure on the category of diagrams over 
$\cM$ with objectwise weak equivalences and fibrations.

(ii) If $\bfD$ is an inverse category, there is a 
diagram-injective model structure on the category of diagrams over 
$\cM$ with objectwise weak equivalences and cofibrations.

(iii) In both  the direct and the inverse case, if each model
structure appearing in the diagram is cellular and proper, then so are
the model structures on the category of diagrams. \qqed
\end{prop}

We will write $\mcL_{\bfD}\cM $ for the category of
$\cM$-diagrams with  the diagram injective model structure,
since it is Bergner's  \emph{lax homotopy limit} of the diagram of
model categories~\cite{Bergner}. This model structure also appears in the literature as the \emph{diagram of left sections} in the work of Barwick~\cite{MR2771591}. Under mild conditions we can also  describe the corresponding
\emph{strict homotopy limit} of model categories as a certain right Bousfield localisation of this lax limit. However, it may be more helpful to first describe what will be its subcategory of fibrant and cofibrant objects as first observed in the work of To\"{e}n~\cite{HallAlg}.

\begin{defn}
\label{defn:strictlimit}
The {\em cocartesian skeleton}  $\spcLim_{\bfD} X$ of the diagram $\cM$ of
model categories is the subcategory of $\cM$-diagrams which are
objectwise fibrant and cofibrant and the base change maps $\alpha_\ast X(s) \to X(t)$ are weak
equivalences.  That is, we take the subcategory of those left sections which are \emph{homotopy cocartesian}.
\end{defn}

This subcategory should be thought of as the homotopical skeleton of the homotopy limit of model categories. To explain further, there are several models for the homotopy theory of homotopy
theories. We are working in the category of model categories, which is not 
complete. However the category $\textbf{CSS}$  of complete Segal spaces is complete and
 Bergner~\cite{BergnerHom} shows that complete Segal spaces gives a model of the homotopy theory 
of homotopy theories. 

Given a model category $\cC$ we can consider its associated complete Segal space $L_C
\cC$, and hence given a diagram $\cM$ of model categories and left Quillen
functors, we may consider the corresponding diagram $L_C\cM$ of complete
Segal spaces and take the homotopy limit $\mathcal{L}im_{\bfD}(L_C
\cM) \in \textbf{CSS}$.  We aim to define a strict homotopy limit of
model categories so that $\mathcal{L}im_{\bfD}(L_C\cM) \simeq L_C
(\mathcal{L}im_{\bfD} \cM)$. The following theorem uses the machinery
of right Bousfield localisation as in~\cite{Hirschhorn}. We shall also recap the relevant definitions in Section~\ref{sec:rbl}.

\begin{thm}[{\cite[Theorem 3.2]{Bergner},\cite[Theorem 5.25]{MR2771591}}]
\label{thm:strictlim}
Let $\cM $ be a diagram of shape $\bfD$ of model categories and
left Quillen functors between them. If $\mcL_{\bfD} \cM$ is right proper and combinatorial, then there is a model structure  $\cLim_{\bfD}\cM$ on the category of $\cM$-diagrams 
so that there is an equivalence of complete Segal spaces (and therefore of homotopy theories) 
$L_C (\cLim_{\bfD} \cM) \simeq \cLim_{\bfD}(L_C \cM)$.  Moreover, the model structure may be described by saying that the identity is a right Quillen functor 
$$\mcL_{\bfD}\cM \lra \cLim_{\bfD} \cM$$
which is a right Bousfield localisation at the generalized diagrams in which the base
change maps are all weak equivalences. As such, the category of fibrant and cofibrant objects of
$\cLim_{\bfD}\cM$  is the cocartesian skeleton $\spcLim_{\bfD}\cM$. 
\end{thm}

\begin{proof}
As we have assumed each model category is combinatorial, it follows \emph{a fortiori} each $\cM(s)$ appearing in $\cM$ is locally presentable. This property is used to find a set of objects $\mathcal{A}(s)$ which generate each $\cM(s)$ under $\lambda$-filtered colimits for some sufficiently large cardinal $\lambda$. We may choose these objects such that the objects of $\mathcal{A}(s)$ are cofibrant in $\cM(s)$. Given $X(s) \in \mathcal{A}(s)$, and a left Quillen functor $a_\ast \colon \cM(s) \to \cM(t)$, we can consider the class of all objects $X(t) \in \cM(t)$ which are equipped with a weak equivalence $a_\ast X(s) \to X(t)$. We choose a cofibrant replacement of the objects $X(t)$ in $\mathcal{A}(t)$ if possible. However if $X(t) \not\in \mathcal{A}(t)$ then we must add it to the generating set of $\cM(t)$. We repeat this process for all $s \in \bfD$ and all $X(s) \in \mathcal{A}(s)$, and we end up with a potentially larger set of objects $\mathcal{A}(s)^{(1)}$. We repeat this process to get a diagram 
$$\mathcal{A}(s) \to \mathcal{A}(s)^{(1)} \to \mathcal{A}(s)^{(2)} \to \cdots $$
whose colimit we denote $\mathcal{B}(s)$. We then define a set of objects in $\cM$
$$\{(X(s), a_\ast \colon X(s) \to X(t))_{s,t}) \mid X(s) \in \mathcal{B}(s), \; \tilde{X}(a) \colon a_\ast X(s) \to X(t) \text{ weak equivalence in } \mathcal{M}(t) \}.$$
Bergner \cite{Bergner} shows that cellularization at this set of
objects gives the strict homotopy limit model structure: the weak
equivalence to the complete Segal space limit is given
as~\cite[Theorem 5.1]{Bergner}.  It is clear by construction that the
category of fibrant-cofbrant objects exactly coincides with the
cocartesian skeleton. 
\end{proof}

\begin{remark}
Note that the above theorem uses a combinatorial hypothesis which we
do not assume for our Noetherian model categories. It is sometimes
still possible to construct such a strict homotopy limit without the
combinatorial hypothesis by using the compact generators (see Theorem
\ref{mame} below). 
\end{remark}
%
%
%

\subsection{Model structures on module categories}

We recall the necessary definitions and results from~\cite{SchwedeShipleyAlg} which allow us to form module categories in monoidal model categories. Recall that a \emph{monoid} is an object $R \in \cC$ together with a map $R \tensor R \to R$ and a unit $\unit \to R$ which satisfy the obvious associativity and unit conditions. A \emph{left $R$-module} in $\cC$ is an object $N$ together with a map $R \tensor N \to N$ satisfying associativity and unit conditions. We can then construct the category $R \textbf{-mod}_\cC$ of \emph{(left) $R$-modules in $\cC$}. There is a forgetful functor $i^\ast \colon R \textbf{-mod}_\cC \to \cC$.

\begin{prop}[{\cite[Theorem 4.1]{SchwedeShipleyAlg}}]\label{modmodcat}
Let $\cC$ be a cofibrantly generated monoidal model category, and $R$ a cofibrant commutative monoid in $\cC$. Then there is a cofibrantly generated monoidal model structure on $R \textbf{-mod}_\cC$ where a map $f \colon N \to M$ is  
\begin{itemize}
\item A weak equivalence if $i^\ast(f) \colon i^\ast(N) \to i^\ast(M)$ is a weak equivalence in $\cC$. 
\item A fibration if $i^\ast(f) \colon i^\ast(N) \to i^\ast(M)$ is a fibration in $\cC$. 
\item A cofibration if it has the LLP with respect to all acyclic fibrations. \qqed
\end{itemize}
\end{prop}

\section{Localizing model categories at a prime}\label{sec:localize}

Motivated by the commutative algebra in Subsection
\ref{subsec:commalg},  we  construct localization $L_{\fp}$ at a
Balmer prime $\fp$ in model 
categorical terms as fibrant replacement in the nullification of $\fp$, and show it is covariant in the Balmer ordering. 

\subsection{Left Bousfield localization}

We recall the theory of left Bousfield localization using
\cite{Hirschhorn} for reference, and~\cite{BarnesRoitzheim}
for some further properties. We will denote by
$\mathrm{Map}_{\cC}(-,-) \in \textbf{sSet}$ the homotopy function
complex. The idea is to specify a set $S$ of maps in $\cC$ and to
construct  model category by making elements of $S$ into weak equivalences.

\begin{defn}
Let $\cC$ be a model category and $S$ be a set of maps in $\cC$. 
\begin{itemize}
\item An object $Z \in \cC$ is said to be \emph{$S$-local} if
$$\mathrm{Map}_{\cC}(s,Z) \colon \mathrm{Map}_{\cC}(B,Z) \to \mathrm{Map}_{\cC}(A,Z)$$
is a weak equivalence in $\textbf{sSet}$ for any $s \colon A \to B$ in $S$.  

\item A map $f \colon X \to Y$ in $\cC$ is an \emph{$S$-equivalence} if
$$\mathrm{Map}_{\cC}(f,Z) \colon \mathrm{Map}_{\cC}(Y,Z) \to \mathrm{Map}_{\cC}(X,Z)$$
is a weak equivalence for any $S$-local object $Z \in \cC$.

\item An object $W \in \cC$ is \emph{$S$-acyclic} if $\mathrm{Map}_{\cC}(W,Z) \simeq \ast$ for any $S$-local object $Z \in \cC$. 
\end{itemize}
\end{defn}

\begin{defn}
The \emph{left Bousfield localization} of a model category $\cC$
inverting a set of maps $S$ (if it exists) is the model category
$L_{S}\cC$ with underlying category of $\cC$ such that
\begin{itemize}
\item The weak equivalences of $L_S \cC$ are the $S$-equivalences.
\item The fibrations of $L_S \cC$ are those maps with the RLP with respect to the cofibrations which are also $S$-equivalences (i.e., the acyclic cofibrations of $L_S \cC$).
\item The cofibrations of $L_S \cC$ are the cofibrations of $\cC$.
\end{itemize}
The fibrant objects in $L_S \cC$ are exactly the fibrant objects of
$\cC$ which also happen to be $S$-local.  Accordingly the identity
functor $\text{id} \colon \cC \to L_S \cC$ is a left Quillen
functor. Finally, for any object $X$ of $\cC$ we write $X\lra L_SX$
for a functorial fibrant replacement.  
\end{defn}

\begin{prop}[{\cite[Theorem 4.1.1]{Hirschhorn}}]
If $\cC$ is a left proper and cellular model category, then the Bousfield
localization inverting  any given set of maps $S$ exists. \qqed
\end{prop}

Localizations preserve certain  properties as we see
from~\cite{BarnesRoitzheim}. By cofibrant replacement, we may  assume that the set $S$ consists of cofibrations between cofibrant objects. 

\begin{prop}[{\cite[Propositions 3.6, 4.7 and 5.4]{BarnesRoitzheim}}]
\label{prop:locmon}
Let $\cC$ be a proper, cellular, stable monoidal model category and $S$ a set of cofibrations between cofibrant objects. Assume moreover that:
\begin{enumerate}
	\item $S$ is closed under taking $\Sigma$;
	\item $S \square I$ is contained in the class of $S$-equivalences, where $I$ is the class of generating cofibrations and $-\square - $ is the pushout-product.
\end{enumerate}
Then the Bousfield localization $L_S \cC$ is a proper, cellular, stable monoidal model category.\qqed
\end{prop}

We are interested in two particular flavours of Bousfield localization. First, the
nullification of a set of small objects (i.e., the set
of maps inverted are maps $N\lra 0$ from the nullified objects to a
point) gives the localization at a prime discussed in Subsection
\ref{subsec:Locp}. Second, for monoidal model categories, classical Bousfield localization with
respect to an object $E$ (i.e., inverting all maps inducing an
isomorphism in $E$-homology), which we will use in Section
\ref{sec:compp} to give completion at a prime. 

\subsection{Localization at a prime}
\label{subsec:Locp}

Suppose  $\cC$ is an Noetherian model category and $\fp$ is a Balmer prime of $\cCb$. 
Motivated by the commutative algebra in Subsection \ref{subsec:commalg}, 
we take $L_{\fp}$ to be the nullification of $\fp$. More precisely, we note
that $\fp$ has a small skeleton and invert the  set of 
maps $S_{\mathfrak{p}} = \{X \to 0 \mid X \in \mathfrak{p} \}$. 

Recall that a Bousfield localization is called \emph{smashing} if it
preserves homotopy colimits. 

\begin{lemma}\label{lem:ploc}
The localizations $L_\mathfrak{p}$ exist, and are monoidal, stable, and smashing.
\end{lemma}

\begin{proof}
As $\mathfrak{p}$ is a thick-tensor ideal, it is clear that the
localization exists, and is stable. To see that it is monoidal, note
that for an arbitrary map in $S_\fp$, say $X \to 0$, the map $A
\tensor X \to 0$ is also an $S_\fp$ equivalence for any cofibrant object $A$ due to $\fp$ being an ideal. 

To see that the localization is smashing, note, that by construction, it is a finite localization in the sense of~\cite{miller} (indeed the localizing subcategory of $\fp$ is generated by small objects). 
\end{proof}

\subsection{Variance in the prime}
We observe that if $\fq \subseteq \fp$ then $L_{\fp}$ nullifies more
than $L_{\fq}$, so we have left Quillen functors
$$ \cC\lra L_{\fq}\cC\lra L_{\fp}\cC,  $$
where we note that this is \emph{covariant} in the Balmer ordering. 
Indeed if we let $\MCL(\cC)$ denote the category whose objects are
model  structures on $\cC$ with the same set of cofibrations as $\cC$,
and  with morphisms the left Quillen functors then we have a
functor
$$L_{\bullet} \colon \spcc(\cCb)\lra \MCL(\cC). $$

\begin{lemma}
\label{lem:LXdiagram}
Suppose  that the tensor-triangulated category $\cCb$ is finite
dimensional Noetherian. 
For each object $X$ there is a  diagram 
$$L_{\bullet}X \colon \spcc(\cCb)\lra \cC$$
 so that $L_{\fp}X$ is a fibrant replacement of $X$ in
 $L_{\fp}\cC$. If $X$ is a ring, then this may be taken to be a diagram of rings. 
\end{lemma}
\begin{proof}
The idea is to choose fibrant
replacements $L_{\fp}X$ for all $\fp$, and then fill in the maps in
adjacent layers by functoriality, starting with the top
dimension. This may be made precise as follows. We have considered the diagram  
$$L_{\bullet}\cC \colon \spcc(\cCb)\lra \MCL (\cC)$$
of model categories. Krull dimension, viewed as a function on $\spcc(\cCb)$ shows that it is a 
direct  category so Proposition \ref{diamod} shows 
that the category of generalized diagrams admits the  diagram-projective model structure with weak 
equivalences and fibrations 
objectwise. Furthermore, $\spcc(\cCb)$ is a poset, so since we start with a diagram of proper, 
cellular model categories, the resulting category of diagrams is also 
proper and cellular.  Accordingly if we take the fibrant replacement of the 
constant diagram $X$ then we obtain a diagram $L_{\bullet}X$.

To show we have a diagram of rings when $X$ is a ring, we use a different model
structure,  constructed directly as a monoidal Bousfield
localization. We start with the diagram of shape
$\spcc(\cCb)$ constant at the model category $\cC$. We equip this with the
injective model structure, with weak equivalences and cofibrations
determined objectwise. 

Now consider the class of maps of diagrams in which the component at
$\fp$ is required to lie in $S_{\fp}$:
$$S = \{f=(f_{\fp})_{\fp} \mid f_\fp \in S_\fp \; \mbox{ for all } \fp \in \spcc(\cCb)
\}.$$ 
We consider the left Bousfield localization  $L_S$ inverting  $S$.
 It is easy to see that $S$ satisfies the second condition of
Proposition~\ref{prop:locmon}. Indeed, if $A$ is a  cofibrant diagram 
(i.e.,  $A_\fp$ is cofibrant for all $\fp$), then  by
Lemma~\ref{lem:ploc}, $A_\fp \tensor f_\fp$ is an $S_\fp$ equivalence
and $A \tensor f$ is in $S$. Using this and the properness, it follows
that the pushout product map lies in $S$ as required. 

It remains to observe that if $\eta \colon X\lra L_SX$ is a fibrant
replacement in the $S$-localization 
then  for each prime $\fp$ the map $X_{\fp}\lra (L_SX)_{\fp}$ is also fibrant replacement in the
$S_{\fp}$-localization, so that $(L_SX)_{\fp}=L_{S_{\fp}}(X_{\fp})$. This argument is quite general, and only uses the fact that
the sets $S_{\fp}$ increase with $\fp$ (in our case this is clear
since if $\fq\subseteq \fp$ then nullifying objects of $\fp$ nullifies
objects in $\fq$). 

Firstly,  we show that if $f \colon X\lra Y$ is an $S$-equivalence, then
$f_{\fp} \colon X_{\fp}\lra Y_{\fp}$ is an $S_{\fp}$-equivalence. The right adjoint 
to evaluation at $\fp$ is the `constant below $\fp$' functor. 
Indeed,  for any object $Y_{\fp}$ we may 
consider the diagram $\Lambda (\fp, Y_{\fp})$ which is constant at 
$Y_{\fp}$ on  $\Lambda (\fp)$ and 0 otherwise. Then 
$$\map(X, \Lambda (\fp, Y_{\fp}))=\map (X_{\fp}, Y_{\fp}).$$
It follows that if $Y_{\fp}$ is $S_{\fp}$-local then $\Lambda(\fp, Y_{\fp})$ is 
$S$-local. Testing $f$ against the $S$-local objects 
$\Lambda (\fp, Y_{\fp})$ we see that   $f_{\fp} $ is an
$S_{\fp}$-equivalence. In particular this applies to $f=\eta$. 

Secondly,  we show that if $Z$ is $S$-local then $Z_{\fp}$ is
$S_{\fp}$-local.  The left adjoint 
to evaluation at $\fp$ is the `constant above $\fp$' functor. Indeed,
for any object $A_{\fp}$ we may define $V (\fp, A_{\fp}) $ to be constant at $A_{\fp}$ on
$V(\fp)$ and zero elsewhere and then 
$$\map(V (\fp, A_{\fp}), X )=\map (A_{\fp}, X_{\fp}).$$
If $s_{\fp} \colon A_{\fp}\lra B_{\fp}$ lies in $S_{\fp}$ then by the hypothesis that
$S_{\fp}$ increases with $\fp$ it follows that $s_{\fp} \colon A_{\fp}\lra B_{\fp}$
lies in $S_{\hat{\fp}}$ whenever $\fp \subseteq \hat{\fp}$ and hence
$V (\fp, s_{\fp}): V(\fp, A_{\fp})\lra V(\fp, B_{\fp})$ lies in $S$. Now testing $Z$ against
$V (\fp, s_{\fp})$ is testing $Z_{\fp}$ against $s_{\fp}$, so
$Z_{\fp}$  is $S_{\fp}$-local as required. This applies in particular
to $L_SX$.  
\end{proof}

\begin{remark}
As far as the formal proof is concerned, we need never have mentioned
the diagram-projective model structure, but we consider it useful
motivation. We note that the identity functor from the diagram-projective model structure to
the $S$-local model structure is a left Quillen functor. 
\end{remark}

\section{Cellularizing and completing model categories at a prime}
\label{sec:compp}

Motivated by commutative algebra in Subsection \ref{subsec:commalg}, 
we  construct completion $\Lambda_{\fp}$ at a Balmer prime $\fp$ in
model theoretic terms and show this is contravariant in the Balmer ordering. 
Indeed,  $\Lambda_{\fp}$ is the fibrant replacement in the 
 $K_{\fp}$-localization of $\cC$. Furthermore, we can construct this fibrant
 replacement by cellularizing the unit to form $\Gamma_{\fp}\unit$ and
 then taking $\Lambda_{\fp}X:=\Hom (\Gamma_{\fp}\unit, X)$. 
%
%

\subsection{Classical Bousfield localization}
For monoidal model categories we may define homology the $E$-homology
for an object  $E$ the $E$-homology by $E_*(X)=[\unit,
E\tensor X]_*$.  The classical Bousfield $E$-homology localization (if it
exists) is the localization which inverts the set of $E$-homology equivalences.
In favourable circumstances, 
Bousfield's method \cite{Bousfield}  identifies a generating set $T_E$ of maps such that the
$T_E$-equivalences will coincide with the $E$-equivalences, and
establishes the existence of the localization. 

Up to equivalence, we may assume that  $E$  is a cell object in $\cC$ (with respect to the generating cofibrations). Let $X$ be a cell object, and denote by $\# X$ the cardinality of the set of cells of $X$.  We then fix an infinite cardinal $c$ which is at least the cardinality of $\operatorname{max}(\#(E \tensor G))$ for $G \in \mathcal{G}$, a compact generator of $X$.  We then let $T$ be the set of $E$-acyclic inclusions of subcomplexes in cell objects $Y$ such that $\#Y \leq c$.  We then have that $T$ is a test set for the $E$-fibrations. Note that the maps in $T$ are then cofibrations between cofibrant objects.

Using the arguments from \cite[\S VIII.1]{EKMM}, we can conclude the
following in the case that $\cC$ is a sufficiently well behaved model category.

\begin{cor}
Let $E$ be an object of a rigidly compactly generated model category
$\cC$, then there is a left Bousfield localization of $\cC$, denoted
$L_E\cC$, so that the  weak equivalences are the $E$-equivalences and
the cofibrations are the cofibrations in the original model category
$\cC$. This localization is stable and monoidal. 
\end{cor}

\subsection{Completion at a prime}
We suppose that the prime $\fp$ is visible, so that there is a small
object $K_{\fp}$ with support $\Lambda (\fp)$. By Lemma  \ref{lem:Kfp}, any two such
small objects generate the same thick tensor ideal of small 
objects, so the following construction does not depend on this 
choice.  By Lemma \ref{lem:noeth},  the objects $K_{\fp}$ exist for all primes if $\spcc(\cCb)$
is Noetherian.  

\begin{defn}
If $\fp$ is visible, 
we define {\em completion at $\fp$} by $\Lambda_{\fp}=L_{K_{\fp}}$. 
Accordingly we have a left
Quillen functor $\cC \lra \Lambda_{\fp}\cC$ and a fibrant replacement 
$X\lra \Lambda_{\fp}X$ for any object $X$.
\end{defn}

Completion at $\fp$  is usually not smashing and it need not preserve small
objects. Even the monoidal unit may fail to  be small in $\Lambda_{\fp} \cC$.

\subsection{Variance in the prime}
\label{subsec:Lambdavar}
We observe that if $\fq \subseteq \fp$ then the classification of 
thick categories shows that $\thick(K_{\fq})\subseteq \thick 
(K_{\fp})$ and therefore every $K_{\fp}$-equivalence is a 
$K_{\fq}$-equivalence.  Accordingly we have left Quillen functors 
$$ \cC\lra \Lambda_{\fp}\cC\lra \Lambda_{\fq}\cC,  $$
where we note that this is \emph{contravariant} in the Balmer ordering. 
Indeed if we let $\MCL(\cC)$ denote the category whose objects are 
model  structures on $\cC$ with the same set of cofibrations as $\cC$, 
and  with morphisms the left Quillen functors then we have a 
functor 
$$\Lambda_{\bullet} \colon \spcc(\cCb)^{op}\lra \MCL(\cC). $$

\begin{lemma}
\label{lem:Lambdafunctor}
Suppose  that the tensor-triangulated category is finite dimensional Noetherian. 
For each object $X$ there is a  diagram 
$$\Lambda_{\bullet}X \colon \spcc(\cCb)^{op}\lra \cC$$ 
so that $\Lambda_{\fp}X$ is a fibrant replacement of $X$ in
$\Lambda_{\fp}\cC$. If $X$ is a ring, this is a diagram of rings. 
\end{lemma}
\begin{proof}
The idea is to choose fibrant replacements $\Lambda_{\fp}X$ for all
$\fp$, and then fill in the maps in adjacent layers by functoriality,
starting with the closed points. This may be made precise as follows.
 We have considered the diagram  
$$\Lambda_{\bullet}\cC \colon \spcc(\cCb)^{op}\lra \MCL (\cC)$$
of model categories. Krull dimension, viewed as a  function on $\spcc(\cCb)$ shows
that it is an inverse category so its opposite is direct, and 
Proposition \ref{diamod} shows that the generalized diagram 
category  admits the  diagram-projective model structure with weak 
equivalences and fibrations 
objectwise. Furthermore, $\spcc(\cCb)$ is a poset, so that since we start with a diagram of proper, 
cellular model categories, the resulting category of diagrams is also 
proper and cellular. 

This in turn means that if we take the fibrant replacement of the 
constant diagram $X$ then we obtain a diagram $\Lambda_{\bullet}X$. 

To see that this is a diagram of rings we apply the same argument as in
Lemma~\ref{lem:LXdiagram}. Noting that completion is contravariant in
the prime, we need only observe that if $S_{\fp}$ consists of
$K_{\fp}$-equivalences,  then $S_{\fp}$ increases as the prime gets
smaller.  Indeed $K_{\fp}$-equivalences are the $\Lambda
(\fp)$-equivalences and if $\fq\subseteq \fp$ then $\Lambda (\fq)
\subseteq \Lambda (\fp)$.
\end{proof}

\subsection{Cellularization}\label{sec:rbl}
We recall the definition of cellularization (or right Bousfield localization). 

\begin{defn}
Let $\cC$ be a model category and $\cK$ a set of objects of $\cC$.
\begin{itemize}
\item A map $f \colon A \to B$ in $\cC$ is a \emph{$\cK$-coequivalence} if 
$$\mathrm{Map}_{\cC}(X,f) \colon \mathrm{Map}_{\cC}(X,A) \to \mathrm{Map}_{\cC}(X,B)$$
is a weak equivalence in $\textbf{sSet}$ for each $X \in \cK$.
\item An object $Z \in \cC$ is \emph{$\cK$-colocal} if
$$\mathrm{Map}_{\cC}(Z,f) \colon \mathrm{Map}_{\cC}(Z,A) \to \mathrm{Map}_{\cC}(Z,B)$$
is a weak equivalence for any $\cK$-coequivalence.
\item An object $A \in \cC$ is \emph{$\cK$-coacyclic} if $\mathrm{Map}_\cC(W,A) \simeq \ast$ for any $\cK$-colocal object $W \in \cC$.
\end{itemize}
\end{defn}

\begin{defn}
The \emph{$\cK$-cellularization} (or right Bousfield
localization at $\cK$) of $\cC$ is the model category $\cell_{\cK} \cC$
with underlying category of $\cC$ such that
\begin{itemize}
\item The weak equivalences of $\cell_{\cK} \cC$ are the $\cK$-coequivalences.
\item The fibrations of $\cell_{\cK} \cC$ are the fibrations of $\cC$.
\item The cofibrations of $\cell_{\cK} \cC$ are those maps with the LLP with respect to the fibrations which are also $\cK$-coequivalences (i.e., the acyclic fibrations of $\cell_{\cK} \cC$).
\end{itemize}
The cofibrant objects in $\cell_{\cK} \cC$ are exactly the
cofibrant objects of $\cC$ which also happen to be $\cK$-colocal.  The
identity functor $\text{id} \colon \cC \to \cell_{\cK} \cC$ is a
right Quillen functor. Finally, for any object $X $ of $\cC$ we write
$\cell_{\cK}X\lra X$ for a functorial cofibrant replacement. 
\end{defn}

\begin{prop}[{\cite[Theorem 5.1.1]{Hirschhorn}}]
If $\cC$ is a right proper and cellular model category, then cellularization exists at any set of objects $\cK$.\qqed
\end{prop}

\subsection{Cellularization at a prime}
As before, Balmer's classification shows that any two objects $K_{\fp}$ with 
support $\Lambda (\fp)$ generate the same thick tensor ideal of small 
objects, so the following construction does not depend on this 
choice. 

\begin{defn}
Given a visible prime $\fp$ we write $\Gamma_{\fp}=\cell_{K_{\fp}}$ for the
$K_{\fp}$-cellularization, so that for an object $X$  of $\cC$ there
is a map $\Gamma_{\fp}X\lra X$ which is a $K_{\fp}$-coequivalence 
 from a $K_{\fp}$-cellular object. Indeed, we have a right Quillen functor
$$\cell_{K_{\fp}}\cC\lra \cC$$
and we may take $\Gamma_{\fp}X$ to be the cofibrant approximation of
$X$   in $\Gamma_{\fp}\cC$.
\end{defn}

In our situation this is a rather concrete construction: since
$K_{\fp}$ is small we may construct $\Gamma_{\fp}X$ up to equivalence
by repeatedly attaching cells $K_{\fp}$ and passing to sequential colimits. 

\subsection{Variance in the prime}\label{subsec:comvar}
We observe that if $\fq \subseteq \fp$ then the classification of 
thick categories shows that $\thick(K_{\fq})\subseteq \thick 
(K_{\fp})$ and therefore every $K_{\fp}$-coequivalence is a 
$K_{\fq}$-coequivalence.  Accordingly we have right Quillen functors 
$$ \Gamma_{\fq}\cC\lra \Gamma_{\fp}\cC\lra \cC,  $$
where we note that this is \emph{covariant} in the Balmer ordering. 
Indeed if we let $\MCR(\cC)$ denote the category whose objects are 
model  structures on $\cC$ with the same set of fibrations as $\cC$, 
and  with morphisms the right Quillen functors then we have a 
functor 
$$\Gamma_{\bullet} \colon \spcc(\cCb)\lra \MCR(\cC). $$

\begin{lemma}
Suppose  that the tensor-triangulated category is Noetherian. 
For each object $X$ there is a  diagram 
$$\Gamma_{\bullet}X \colon \spcc(\cCb)\lra \cC$$ 
so that $\Gamma_{\fp}X$ is a cofibrant replacement of $X$ in $\Gamma_{\fp}\cC$.
\end{lemma}
\begin{proof}
In effect, we choose cofibrant 
replacements $\Gamma_{\fp}X$ for all $\fp$, and then fill in the maps in 
adjacent layers by functoriality, starting with the top dimension. 
This may be made precise as follows. We have considered the diagram  
$$\Gamma_{\bullet}\cC \colon \spcc(\cCb)\lra \MCR (\cC)$$
of model categories. The dimension function on $\spcc(\cCb)$ shows that it is a 
inverse  category so Proposition \ref{diamod} shows 
that it admits a diagram-injective model structure with weak
equivalences and cofibrations 
objectwise. Furthermore, $\spcc(\cCb)$ is a poset, so  that since we start with a diagram of proper, 
cellular model categories, the resulting category of diagrams is also 
proper and cellular. 

This in turn means that if we take the cofibrant replacement of the 
constant diagram $X$ then we obtain a diagram $\Gamma_{\bullet}X$. 
\end{proof}

\subsection{Completion at a prime via cellularization of the unit}
The connection between local cohomology and completion has a natural
counterpart in the general context. 

\begin{lemma}
For an object $X$ fibrant in $\cC$, 
the map $X\lra \Hom(\Gamma_{\fq}\unit, X)$ is a
$\Lambda_{\fq}$-fibrant approximation. More informally
$$\Lambda_{\fq}X\simeq \Hom(\Gamma_{\fq}\unit, X).\qqed $$ 
\end{lemma}

\subsection{Two extremes}\label{subsec:extremes}
We finish by recording two different results regarding the behaviour of the completion at the maximal and minimal elements of the Balmer spectrum.

\begin{lemma}\label{red1}
If $\spcc(\cC)$ is irreducible (in the sense that it has a unique
maximal point, the generic point  $\gen$), then 
$ X \stackrel{\simeq}\to  
 \Lambda_{\gen} X $ is an equivalence.
\end{lemma}

\begin{proof}
By definition we have $\overline{\{\gen \}} = \spcc(\cC)$, and we may
take $K_{\gen}=\unit$.  The localization therefore only inverts those
maps that are already  weak equivalences.
\end{proof}

\begin{lemma}\label{red2}
Let $\fm$ be a Balmer minimal prime (closed point), then 
$\Lambda_\fm X\stackrel{\simeq}\to L_\fm \Lambda_\fm X$
 is  an equivalence. 
\end{lemma}

\begin{proof}
We need to show that nullifying  elements of $\fm$ makes no difference
to a completion. Indeed, since $\Lambda_\fm X$ is $K_{\fp}$-local, we
need only observe that if $A\in \fm$ then $A\lra 0$ is a
$K_{\fp}$-equivalence.  However  $A\in \fm$ means $\fm \not \in \supp (A)$
and so $\supp (A\tensor K_{\fm}) \subseteq \supp (A)\cap \supp(
K_{\fm})=\emptyset$.  Accordingly $A\tensor K_{\fm}\simeq 0$ and
$A\lra 0$ is a $K_\fm$-equivalence as required. 
\end{proof}

\section{The adelic cube}
\label{sec:adeliccube}

In \cite{adelicc}, a certain `adelic  cochain complex' $C^*_{ad}(\fX; L,\cF)$ was
constructed. Taking the relevant special case, this takes as input
\begin{itemize}
\item a poset $\fX$ with a dimension function;
\item a coefficient system $\cF \colon \fX^{op}\lra \cA$;
\item a compatible system of localizations $L \colon \fX \lra [\cA, \cA]$. 
\end{itemize}
We will recall the definitions below, but  we may take $\fX=\spcc(\cCb)$, which we assume to be of
finite dimension $r$ equipped with the Krull dimension. The motivating example
has $\cCb=\sfD (R)$ for a finite dimensional commutative Noetherian ring, $\cA$ is the category of $R$-modules
$\cF(\fp)=M_{\fp}^{\wedge}$ ($\fp$-adic
completion of a fixed  $R$-module $M$) and
$L_{\fp}M=M_{\fp}$ (localization at $\fp$). 

The purpose of this section is to describe the counterpart at the
model category level, and the reader wanting precision should
immediately skip over the next two short paragraphs sketching the vision. 

 The idea is that an object can be assembled from
information at each prime $\fp$. If we start from localized completed 
objects, we can hope to retain good multiplicative information. The
data for an object
should be thought of as being based on the completed information 
at closed points, with additional structure added to its
localizations so as to describe the whole. 

To understand the formal structure used to assemble this, one might wish to think of having a
local coefficient system of rings on the simplicial complex of the
Balmer spectrum poset. To convert to algebra one takes
`locally finite sections'. The model then consists of a category of
modules over this local system of rings.

\subsection{Objects of the adelic diagram}
The construction at the abelian category level in \cite{adelicc} may be copied verbatim
to our present context.  The adelic diagram is based on the  cube consisting of subsets of $\Delta^r=\{0,1,
\ldots , r\}$, where $r=\dim (\fX)$. The diagram is a functor on the  punctured
cube  $(\Delta^r)'$ of nonempty subsets. For an object $M$ in $\cC$,
we will define the adelic diagram
$$M_{ad} \colon (\Delta^r)'\lra \cC. $$
(in the notation of \cite{adelicc} we have $M_{ad}:=C^{\bullet}_{ad}(\fX; L, \Lambda M)$). 
The value on a  nonempty subset $\boldd=(d_0>d_1>\cdots >d_s)$ is
obtained from the localized completed cobjects
$L_{\fp_s}\Lambda_{\fp_s}M$ with $\dim (\fp_s)=d_s$ by taking iterated products
and localizations as follows: 
\begin{multline*}
M_{ad}(\boldd):=
\prod_{\dim(\fp_0)=d_0} L_{\fp_0}
\left( \prod_{\dim(\fp_1)=d_1,   \fp_1\subset \fp_0} L_{\fp_1}\right.
\cdots \\
\cdots \left. \left( \prod_{\dim(\fp_{s-1})=d_{s-1}, \fp_{s-1}\subset \fp_{s-2}} L_{\fp_{s-1}}
\left(\prod_{\dim(\fp_s)=d_s, 
  \fp_s\subset \fp_{s-1}}
L_{\fp_s}\Lambda_{\fp_s}M\right)\right)\cdots \right) . 
\end{multline*}
In order to fill in the edges of the punctured cube and make it 
commutative we need $\Lambda_{\bullet}M$ to be a functor as 
shown in Lemma \ref{lem:Lambdafunctor}, and $L$ to consist of functors
with natural transformations $\eta \colon id \lra L_{\fp}$. We write
$M(\fp)=\Lambda_{\fp}M$ to emphasize the (contravariant) functoriality
in $\fp$.

\subsection{Morphisms in the adelic diagram}
Suppose we consider a flag $\boldd=(d_0>d_1>\cdots >d_s)$ and the edge of the
cube $\partial_i \boldd\stackrel{\delta_i}\lra \boldd$ corresponding to the face omitting
$i$. 
If  $i<s$, we may  take 
$$M_{i+1}(\fp_{i}):=
\prod_{\fp_{i+1}<\fp_i }L_{\fp_{i+1}} \prod_{\fp_{i+2}<\fp_{i+1}}L_{\fp_{i+2}}\cdots 
L_{\fp_{s-2}}\prod_{\fp_{s-1}<\fp_{s-2}}L_{\fp_{s-1}}\prod_{\fp_{s}<\fp_{s-1}}L_{\fp_{s}}M(\fp_{s}). $$
Then the map  is simply given by taking 
$$M_{i+1}(\fp_{i-1})\lra M_{i+1}(\fp_{i})\lra
 L_{\fp_i}M_{i+1}(\fp_i)$$
 at the $i$th spot (where the first map uses functoriality in the
 prime and collapses the factors
 corresponding to primes $\fp_{i+1}$ containing $\fp_{i-1}$ but not
 $\fp_i$, and the second is the unit of localization) and then applying the same sequence of products and
localizations to both domain and codomain. 

If $i=s$ we take the map
$$M(\fp_{s-1}) \lra \prod_{\fp_{s}< \fp_{s-1}}M(\fp_{s})\lra \prod_{\fp_{s}< \fp_{s-1}}L_{\fp_{s}}M(\fp_{s})$$
with components $M(\fp_{s-1})\lra M(\fp_{s})\lra L_{\fp_{s}}M(\fp_{s})$ given by the coefficient
system, and then apply $L_{\fp_{s-1}}$ and the same sequence of products and
localizations to both domain and codomain. 

To see we get a cochain complex we need only observe that the
composite of two $\delta_i$s depends only on the dimensions omitted. 
More precisely,  if the numbers omitted are $0\leq a<b\leq s$, then we may omit $a$ and $b$
in either order and we need to know that $\delta_a\delta_b=\delta_{b-1}\delta_a$. 

If $a<b<s$ then the verification is immediate from the fact that
$L_{\bullet}$ is a functor on the diagram category, together with the categorical properties
of the product. 

If $b=s$ there are two cases. The simplest is when  $a<s-2$. Then the diagram 
$$\xymatrix{
L_{\fp_{a+1}}\cdots L_{\fp_{s-1}}M(\fp_{s-1}) \rto \dto&
L_{\fp_{a+1}}\cdots L_{\fp_{s-1}}L_{\fp_s}M(\fp_{s})  \dto \\
L_{\fp_a} L_{\fp_{a+1}}\cdots L_{\fp_{s-1}}M(\fp_{s-1}) \rto &
L_{\fp_a}L_{\fp_{a+1}}\cdots L_{\fp_{s-1}}L_{\fp_s}M(\fp_{s})
}$$
commutes since $\eta: 1\lra L_{\fp_a}$ is a natural
transformation. The required commutation then  follows from the categorical 
properties of the product. 
  
The case $b=s, a=s-1$ is the most complicated. We will abbreviate
$M(\fp_s)=M(s)$ and $L_{\fp_s}=L_s$ for readability. The following
diagram has $L_0L_1\ldots L_{s-2}$ applied to it. 
$$\xymatrixrowsep{3ex}\xymatrix{
&M(s-2)\ar[dl] \ar[drr] \ar[ddd]&&\\
L_{s-1}M(s-2)\ar[drr]\ar[ddd]&&&L_sM(s-2)\ar[dl] \ar[ddd]\\
&&L_{s-1}L_sM(s-1)\ar[ddd]&\\
&M(s-1)\ar[dl] \ar[drr] \ar[ddd]&&\\
L_{s-1}M(s-1)\ar[drr]\ar[ddd]&&&L_sM(s-1)\ar[dl] \ar[ddd]\\
&&L_{s-1}L_sM(s-1)\ar[ddd]&\\
&M(s)\ar[dl] \ar[drr] &&\\
L_{s-1}M(s)\ar[drr]&&&L_sM(s)\ar[dl] \\
&&L_{s-1}L_sM(s)&}$$
The left and right faces commute since the unit for $L_{s-1}$ is a natural 
transformation.
The front and back faces commute since the unit for $L_{s}$ is a natural 
transformation.
The top and bottom faces commute because of the natural transformation
$id \lra L_{\fp}$. 
The relevant square involves $M(s-2), L_{s-1}M(s-1), L_sM(s)$ and
$L_{s-1}L_sM(s)$.  The required commutation then  follows from the categorical 
properties of the product.

\begin{remark} {\em (Finite number of primes)}
We observe that if there are only a finite number of primes, since the
localization commutes with finite products, we have 
$$M_{ad}(d_0>\cdots >d_s)\simeq \prod_{\fp_0\subset \cdots \subset \fp_s, \dim (\fp_i)=d_i} L_{\fp_0} \Lambda_{\fp_s}M. $$ 
\end{remark}

\section{The adelic approximation theorem}\label{sec:aat}

We now come the main ingredient in constructing the adelic model. It
states that the monoidal unit $\unit$ can be reconstructed from
localizations of completions. We use the dimension filtration to give us
an inductive approach to proving an equivalanece and the finite
dimensionality to show this terminates at a finite stage.  The
Noetherian condition is used so that we can work with a diagram
without the need to consider continuity of various constructions in
the prime. 

\subsection{Statement of result}

Let $\Delta^r = \{0,1,\dots ,r\}$ and write $(\Delta^r)'$ for the
poset of its non-empty subsets, which we may think of as a punctured
$(r+1)$-cube.

\begin{thm}[Adelic Approximation]\label{aat}
Let $\cC$ be a Noetherian model category with $r$-dimensional Balmer
spectrum. Then $\unit \in \cC$ is the homotopy pullback of the
punctured  $(r+1)$-cube $\unitad \colon (\Delta^r)' \to \cC$ where the object at position $(d_0 > d_1 > \cdots > d_s)$ is
$$\unitad (d_0 > d_1 > \cdots > d_s) = \prod_{\dim \fp_0 = d_0}
L_{\fp_0} \prod_{\substack{\fp_1 \subset \fp_0  \\ \dim \fp_1 = d_1}}
L_{\fp_1}  \cdots \prod_{\substack{\fp_{s-1} \subset \fp_{s-2}  \\
    \dim \fp_{s-1} = d_{s-1}}} L_{\fp_{s-1}}  \prod_{\substack{\fp_{s}
    \subset \fp_{s-1}  \\ \dim \fp_{s} = d_{s}}} L_{\fp_{s}}
\Lambda_{\fp_s} \unit, $$
and the maps are as described in Section \ref{sec:adeliccube}. 
\end{thm}
\begin{example}\label{runex6}
In the special case when $\cCb$ is the derived category of abelian
groups the statement of the adelic approximation theorem is that the
Hasse square 
$$\xymatrix{
\bZ \ar[r] \ar[d] & \bQ \ar[d] \\ \prod_p \bZ_p^\wedge \ar[r] & \bQ \tensor \prod_p \bZ_p^\wedge \rlap{ ,}
}$$
is a homotopy pullback. 
\end{example}

\begin{example}\label{ex:cube}
When the Balmer spectrum is two-dimensional and irreducible, the
diagram of the adelic approximation theorem states that the cube 
$$\xymatrixcolsep{0ex}\xymatrixrowsep{5ex}\xymatrix{
 &  \displaystyle{\prod_{\fp_1}} L_{\fp_1} \Lambda_{\fp_1}\unit   \ar[rr] \ar[dd]|\hole&&   L_{\fp_0} \displaystyle{\prod_{\fp_1 \subset \fp_0}} L_{\fp_1} \Lambda_{\fp_1}\unit \ar[dd]  \\
\unit \ar[rr] \ar[ur] \ar[dd]&& L_{\fp_0} \unit \ar[dd] \ar[ur]  &\\
&   \displaystyle{\prod_{\fp_1}} L_{\fp_1} \displaystyle{\prod_{\fp_2 \subset \fp_1}}  \Lambda_{\fp_2}\unit  \ar[rr]|-\hole&&   L_{\fp_0} \displaystyle{\prod_{\fp_1 \subset \fp_0}} L_{\fp_1}   \displaystyle{\prod_{\fp_2 \subset \fp_1}}  \Lambda_{\fp_2}\unit \\
 \displaystyle{\prod_{\fp_2}}  \Lambda_{\fp_2}\unit \ar[rr] \ar[ur] &&    L_{\fp_0} \displaystyle{\prod_{\fp_2 \subset \fp_0}}  \Lambda_{\fp_2}\unit  \ar[ur]&
}$$
is a homotopy pullback, where $\fp_i$ runs through primes of dimension $i$.
\end{example}

\subsection{Strategy}

First we recall the Cubical Reduction Principle for homotopy
pullbacks.  A cubical diagram  $X\colon C\lra \bbD$ is a homotopy pullback if the initial
point $X(\emptyset)$ is the homotopy inverse limit over the punctured cube
$PC$. It is thus clear that a 0-cube is a homotopy pullback if
$X(\emptyset)\simeq *$. 
For a 1-cube $X \colon \bbI \lra \mcD$ (with $\bbI=(0\lra 1)$), we write $X_f=\fibre (X(0)\lra
X(1))$ for the homotopy fibre.  This diagram is a   homotopy 
pullback if and only if the map $X(0)\stackrel{\simeq}\lra X(1)$ is an
equivalence which happens if and only if $X_f\simeq *$. 
  
Now suppose $C =\bbI \times C'$, and note that $X \colon C \lra \bbD$ induces
a cube $X^1_f \colon C'\lra \bbD$ of homotopy fibres, where the $1$ refers
to the fact that the fibre has been taken with respect to the first
coordinate.  The Cubical Reduction Principle states that 
 the diagram  $X$ is a homotopy pullback if and only if $X_f^1$ is a
 homotopy pullback.

\begin{proof}[of \ref{aat}]
For each $n$-dimensional prime we may consider the set $\Lambda (\fp)$ 
of primes below $\fp$, and form the $(n+1)$-cube indexed by subsets of 
$\{ 0, 1, \ldots, n\}$. We consider the $(n+1)$-cube $\unitad (\fp)
$, with the same definition as $\unitad$, but the primes are
restricted to $\Lambda (\fp)$ and hence the dimensions are restricted 
to $\{0, 1, \ldots, n\}$. Evidently if $\fq \subseteq \fp$ we have maps of diagrams
$$\unitad(\fq)\lra \unitad(\fp)\lra \unitad.$$

Note  that $\unitad$ is a homotopy pullback if and only if
$K_{\fp}\tensor \unitad $ is a pullback for all $\fp$ by Lemma \ref{lem:Kfp}. Since $K_{\fp}\tensor
\unit_{\fq}\simeq 0$ unless $\fq\subseteq \fp$ we see 
$$K_{\fp}\tensor_R \unitad \simeq K_{\fp}\tensor_R\unitad (\fp), $$
so that it suffices to show $K_{\fp} \tensor_R \unitad (\fp)$ is a pullback
for all primes $\fp$.

We will prove by induction that if $\dim (\fp)=n$ then $\unitad (\fp)$ is a homotopy
pullback in dimension $\leq n$. The base of the induction is the trivial case $n=-1$. 

For the inductive step we suppose that $\dim (\fp)=n$ and if $\fq\subseteq
\fp$ with $\dim (\fq)=i \leq n-1$  then $\unitad (\fq)$ is a homotopy
pullback in dimension $\leq i$. By the Cubical Reduction Principle, 
$\unitad (\fp)$ is a homotopy pullback  if and only if $(\unitad (\fp)
)_f^{n}$  is a homotopy pullback. 

Since $\fp$ is the only $n$-dimensional prime in $\unitad(\fp)$, the
cubical reduction takes the fibre of localization $L_{\fp}=L_{V(\fp)}$
(see Section \ref{sec:fil}), and in
view of  the fibre sequence $\Gamma_{V(\fp)^c}\unit \lra \unit \lra
L_{V(\fp)}\unit$ we have 
$$\unitad (\fp)^{n}_f(d_0>\cdots >d_s)=(\Gamma_{V(\fp)^c}R)\tensor_R \left[\unitad (\fp) (d_0>\cdots 
>d_s)\right]. $$
Any prime $\fq \subseteq \fp$ of dimension $\leq n$ in $V(\fp)^c$ is actually of
dimension $\leq n-1$.  Next note that 
$$K_{\fq}\tensor \unitad (\fp) (d_0>\cdots >d_s)\simeq 0$$
unless $\dim(\fq) \geq d_0$: this uses the fact that $K_{\fq}\tensor \unit_{\fq_0}\simeq 0$
unless $\fq_0\leq \fq$, and the fact that  $K_{\fq}$ is small so
that it passes inside the products. Accordingly, 
 $$K_{\fq}\tensor \unitad (\fp)^{n}_f \simeq K_{\fq}\tensor \unitad
 (\fq),  $$
which is a pullback cube by the induction hypothesis, completing
the inductive step.

By induction we see that $K_{\fp}\tensor_R \unitad (\fp) $ is a homotopy pullback for all
primes of dimension $r$, and hence $\unitad$ is a homotopy pullback as
required. 
\end{proof}

\subsection{The Beilinson-Parshin cube}
The adelic approximation theorem emerged from the algebraic model for 
torus-equivariant rational spectra. Related constructions occur in the 
construction of Beilinson-Parshin 
adeles~\cite{Beilinson,Huber,Morrow}, and there is a corresponding 
statement. 

In fact the inductive scheme of the proof of the Adelic Approximation
Theorem applies equally well to other localization
systems provided $K_{\fp}\tensor A_{\fq}\simeq 0$ unless $\fq \subseteq
\fp$, and provided the support of the fibre of $1\lra A_{\fp}$ does not contain $\fp$.

 If $A_{\fp}=\Lambda_{\fp}L_{\fp}$ as for the
Beilinson-Parshin case the first condition is clear since $K_{\fp}$ is
small and  $K_{\fp}\tensor_R L_{\fq}R\simeq 0$ unless $\fq\leq
\fp$. For the second condition we factor it as $1\lra L_{\fp}\lra
\Lambda_{\fp}L_{\fp}$, and it suffices to show that the fibres of both
factors are supported in dimension $\leq n-1$. This is true as before
for the first map. For the second the fibre is of the form $\Hom (L_{\Lambda (\fp)^c}
R, L_{\fp}M)$, and since $K_{\fp}$ is small and $\fp \not \in \Lambda (\fp)^c\cap
V(\fp)$ its tensor product with $K_{\fp}$ is trivial. 

Now consider the composite functors $A_{\fp}=\Lambda_{\fp}L_{\fp}$,
which is equipped with a natural transformation from the identity. 
  
\begin{thm}[Beilinson-Parshin Adelic Approximation]\label{aatbpv}
Let $\cC$ be a Noetherian model category whose Balmer spectrum is of
topological dimension $r$. Then $\unit \in \cC$ is the homotopy
pullback of the punctured $n$-cube $\unitBP \colon (\Delta^r)'\to \cC$ where the object at position $(d_0 > d_1 > \cdots > d_s)$ is
\small$$\unitBP (d_0 > d_1 > \cdots > d_s) = \prod_{\dim \fp_0 = d_0}
\Lambda_{\fp_0} L_{\fp_0} \prod_{\substack{\fp_1 \subset \fp_0  \\
    \dim \fp_1 = d_1}} \Lambda_{\fp_1} L_{\fp_1}  \cdots
\prod_{\substack{\fp_{s-1} \subset \fp_{s-2}  \\ \dim \fp_{s-1} =
    d_{s-1}}} \Lambda_{\fp_{s-1}} L_{\fp_{s-1}}
\prod_{\substack{\fp_{s} \subset \fp_{s-1}  \\ \dim \fp_{s} = d_{s}}}
\Lambda_{\fp_s} L_{\fp_s} \unit$$
with maps as in Section \ref{sec:adeliccube}. 
\end{thm}
\normalsize

\begin{remark}
If $r=1$ and the Balmer spectrum is irreducible, then (using Lemmas~\ref{red1} and~\ref{red2})  the
Beilinson-Parshin square coincides exactly with the standard adelic square. 
\end{remark}

\begin{example}
Let us rewrite Example~\ref{ex:cube} of the two-dimensional irreducible Balmer spectrum in the Beilinson-Parshin variant:
$$\xymatrixcolsep{0ex}\xymatrixrowsep{5ex}\xymatrix{
 &  \displaystyle{\prod_{\fp_1}} \Lambda_{\fp_1} L_{\fp_1}  \unit   \ar[rr] \ar[dd]|\hole&&   L_{\fp_0} \displaystyle{\prod_{\fp_1 \subset \fp_0}} \Lambda_{\fp_1} L_{\fp_1} \unit \ar[dd]  \\
\unit \ar[rr] \ar[ur] \ar[dd]&&  L_{\fp_0} \unit \ar[dd] \ar[ur]  &\\
&   \displaystyle{\prod_{\fp_1}} \Lambda_{\fp_1} L_{\fp_1} \displaystyle{\prod_{\fp_2 \subset \fp_1}}  \Lambda_{\fp_2}  \unit  \ar[rr]|-\hole&&    L_{\fp_0} \displaystyle{\prod_{\fp_1 \subset \fp_0}}  \Lambda_{\fp_1} L_{\fp_1}   \displaystyle{\prod_{\fp_2 \subset \fp_1}}  \Lambda_{\fp_2} \unit \\
 \displaystyle{\prod_{\fp_2}}  \Lambda_{\fp_2}  \unit \ar[rr] \ar[ur] &&     L_{\fp_0} \displaystyle{\prod_{\fp_2 \subset \fp_0}}  \Lambda_{\fp_2}  \unit  \ar[ur]&
}$$
Notice how the front face is the same as the standard version because of the properties of the completion at the maximal primes and the localizations at the minimal primes (Section~\ref{subsec:extremes}).
\end{example}

\section{The adelic model}\label{sec:model}

The Adelic Approximation Theorem expresses the unit as a pullback of
adelic rings. In this section we show that this in turn  gives a
model for a finite dimensional Noetherian
model category $\cC$ in terms of categories of modules over these adelic
rings.

\subsection{The diagram of adelic module categories} 
 The basic ingredient is the diagram  $\unitad \textbf{-mod}_\cC$ of module
categories. We use the constructions of Subsection \ref{subsec:diagrammodcats} to put
model structures on the corresponding categories of generalized
diagrams. 

\begin{defn}
Let $\cC$ be a model category with $r$-dimensional Noetherian Balmer
spectrum.  We define the punctured $(r+1)$-cube of model categories by
taking the value at  $(d_0 > d_1 \cdots > d_s)$ to be the module
category of the corresponding adelic ring
\small$$\left[\unitad \textbf{-mod}_\cC\right](d_0 > \cdots > d_s) = 
\left[\unitad 
 (d_0 > \cdots > d_s) \right]\textbf{-mod}_\cC$$
\normalsize
equipped with the diagram injective model structure of Proposition~\ref{diamod}. The
morphisms in the diagram are the extension of scalars functors
corresponding to the maps of rings. 
\end{defn}

Proposition~\ref{diamod} shows that the  diagram-injective model structure on 
$\unitad \textbf{-mod}_\cC$ gives a cellular and proper model
category, $\mcL (\unitad \textbf{-mod}_{\cC})$, the lax limit of the
diagram of module categories.  This category of generalized diagrams is related to our original model
category $\cC$ by a Quillen adjunction. Indeed an object $X \in \cC$ gives an object in 
$\unitad \textbf{-mod}_\cC$ by tensoring with the diagram $\unitad$. 
This has right adjoint given by taking the inverse limit over the
diagram. 

\begin{prop}
\label{prop:diagramadj}
There is a Quillen adjunction 
$$\unitad \tensor - : \cC \leftrightarrows \unitad \textbf{-mod}_\cC :
\ilim .\qqed$$
\end{prop}

 We  apply the Cellularization Principle to this Quillen adjunction to
 obtain an equivalence.  We then use the framework of Subsection \ref{subsec:diagrammodcats}
 to give a conceptual reformulation of the result. 

\subsection{$\protect \cC$ is the strict limit of the adelic diagram}
We are now ready to prove our main theorem. 

\begin{thm}{\em (The Adelic Model)}
\label{mame}
Let $\cC$ be a finite dimensional Noetherian model category.  The adjunction of Proposition
\ref{prop:diagramadj} induces a Quillen equivalence
$$ \cC \simeq \cLim \left( \unitad \textbf{-mod}_\cC \right) $$
between $\cC$ and the strict homotopy limit of the diagram
of adelic module categories. Any object is therefore equivalent to one in the
cocartesian skeleton $\spcLim \left( \unitad \textbf{-mod}_\cC \right) $
 (i.e.,  with all base change maps being weak equivalences). 
Moreover, if  $\unitad$ is a diagram of commutative
ring objects,  then the Quillen adjunction is symmetric monoidal.
\end{thm}

\begin{remark}
The same proof applied to  Beilinson-Parshin approximation gives a Quillen equivalence 
$$\cC \simeq_Q\cLim ( \unitBP \textbf{-mod}_\cC). $$
\end{remark}

\begin{proof}
We start by applying the Cellularization Principle~\cite{cellprin} to
show that $\cC$ is equivalent to a cellularization of the category $\mcL
(\unitad\textbf{-mod}_{\cC})$ of generalized diagrams.  We choose a
set $\cG$ of cofibrant small  generators of $\cC$. Accordingly 
$\mathcal{G} \text{-cell-} \cC \cong \cC$~\cite[Proposition 6.2]{cellprin}. 
To apply the Cellularization Principle, we need to show that the objects of  $\cG$ are
small in $\unitad \textbf{-mod}_\cC$ and that the derived unit is an
equivalence.

It is clear that if $X$ is small and  cofibrant in $\cC$, then it is
small and cofibrant in each model category  appearing in the diagram
$\unitad$. Consequently it is small and cofibrant in the diagram
injective model structure on $\unitad \textbf{-mod}_\cC$ as in~\cite[Proposition 6.2]{cellprin}. Therefore,
the images of elements of $\mathcal{G}$ are small and  cofibrant as
required. The fact that the derived unit is an equivalence follows
from the Adelic Approximation Theorem, ~\ref{aat}. The
Cellularization Principle gives  a Quillen equivalence
$$\cC \simeq_Q\mathcal{G}\text{-cell-} \unitad \textbf{-mod}_\cC. $$

All that is left to show is that this cellularization with respect to $\cG$ is the strict
homotopy limit.  
It is clear that every object $g$ of $\cG$ gives an object
$\unitad\tensor g$ of the cocartesian skeleton $\spcLim \left( \unitad \textbf{-mod}_\cC \right)$. Therefore what is left to show is that every object of the cocartesian skeleton can be built from the images of the generators. In particular, it is enough to show that $\imagesofgenerators$ and
$\spcLim \left( \unitad \textbf{-mod}_\cC \right)$ generate the same
localizing subcategory.  This follows from Proposition
\ref{prop:thickCishococart} below. 

The Quillen equivalence being symmetric monoidal under the additional
assumptions follows from~\cite[Proposition 5.1.6]{BGKS}.
\end{proof}

%

\subsection{All homotopy cocartesian modules come from $\protect \cC$}
\label{subsec:reconstruction}
We shall follow the argument from \cite[Theorem 4.5]{tnq1} to prove
the following result.

\begin{prop}
\label{prop:thickCishococart}
The thick subcategory of the homotopy category of $\unitad$-modules obtained from $\cCb$
is precisely the  category of homotopy cocartesian objects. 
\end{prop}

\begin{proof}
We say that an object $X$ of $\cCad$ is \emph{supported in dimension $\leq d$} if
$X(s)\simeq 0$ for $s>d$.
 We show by
induction on $d$ that all homotopy cocartesian objects supported in dimension $\leq d$ are
built from objects in the image of $\cC$. This is clearly true if $d=-1$
since these objects are all contractible. 

Now suppose that objects supported in dimension $\leq d-1$ are built from
objects in the image of $\cC$, and suppose that $X$ is supported in
dimension $\leq d$. We will construct a $\unitad$-module  $f_d(X(d))$
that comes from $\cC$ (and hence is homotopy cocartesian) and is supported
in dimension $\leq d$, and a map 
$$\eta \colon X\lra f_d(X(d)) $$
that is an equivalence at $d$.  It follows  that the mapping cone of 
$\eta$ is supported in dimension $\leq d-1$, and hence built from 
objects in the image of $\cC$ by induction.  Thus $X$ is 
built from the image of $\cC$ as required. An arbitrary object is 
supported in dimension $\leq r$ so after $r$ steps we find all objects 
are built from the image of $\cC$.

\newcommand{\cCaded}{\cC_{ad}^{e(d)}}
It remains to construct $f_d(X(d))$ and the map $\eta$. In fact we
will show that $f_d$ is right adjoint to evaluation at $d$ on a
suitable subcategory of $\unitad$-modules. We suppose given a
$\unitad$-module $M$ and define a $\unitad$-module $f_d(M)$ as
follows. Writing $\bd
=(d_0>d_1>\cdots >d_s)$ for a flag of dimensions, define
$$f_d(M)(\bd)=\trichotomy
{\unitad (\bd)\tensor_{\unitad(d)} M \mbox{ if } d_0=d}
{\unitad (d>\bd)\tensor_{\unitad(d)} M \mbox{ if } d_0<d}
{0 \mbox{ if } d_0>d}.$$
We note that the structure maps of $f_d(M)$ are extensions of scalars
for the maps  of flags $(d)\lra \bd$ if $d_0=d$, and we write $\cCaded$ for the
category of $\unitad$-modules with this property.

\begin{lemma}
The functor $f_d$ is right adjoint to evaluation $ev_d\colon \cCaded \lra \unitad(d)\modules$.
\end{lemma}

\begin{proof} We take the counit $f_d(M)(d)=M\lra M$ to be the
  identity. To define the unit, we suppose $X(d)=M$  and construct $\eta \colon X\lra
f_d(M)$. Indeed, we take the identity at $d$ and use extension of
scalars to define $\eta (\bd)$ when $d=d_0$, as we may do since $X$
lies in $\cCaded$.  There is nothing to do at
$\bd$ if $d_0>d$ since $f_d(M)(\bd)$ is trivial. Finally, if $d>d_0$ the
definition is  determined  by the square 
$$\xymatrix{
X(\bd)\ar@{-->}[r]^-{\eta(\bd)}\dto &\unitad(d>\bd)\tensor_{\unitad(d)}M\dto^{=} \\
X(d>\bd)\rto^-{\eta(d>\bd)}  &\unitad(d>\bd)\tensor_{\unitad(d)}M
}$$
\end{proof}

Applying this when $M=X(d)$ we obtain a map $\eta \colon X\lra f_d(X(d))$. 
By construction $f_d(X(d))$ is supported in dimensions $\leq d$ and
$\eta (d)$ is an equivalence. 

It remains to show that $f_d(X(d))$ comes from $\cC$. For this we note
that all terms of $f_d(\unitad(d))$ are $\unitad(d)$-modules, and 
$$f_d(X(d))=f_d(\unitad(d))\tensor_{\unitad(d)}X(d). $$ 
 The unit of the adjunction gives
an obvious map $\unitad\lra f_d(\unitad(d))$. Tensoring with $X(d)$  gives a map 
$$\phi \colon \unitad\tensor_{\unitad(d)}X(d)\lra f_d(X(d)), $$
which we will show is an equivalence, thereby establishing that
$f_d(X(d))$ comes from $\cC$.

It is clear that $\phi (\bd)$ is an equivalence if $d_0=d$, since the
map is the identity.  Since $X(d)$ is supported in dimensions $\leq d$, it follows that if
$d_0>d$ then $\unitad(\bd)\tensor_{X(d)}X(d)\simeq 0$ and $\phi (\bd) $ is an
equivalence if $d_0>d$. 

Finally we consider $d>d_0$, where we have the map 
$$\phi (\bd) \colon \unitad(\bd)\tensor_{\unitad(d)}X(d)\lra \unitad(d>\bd)\tensor_{\unitad(d)}X(d).$$
Cellularizing at any prime of dimension $<d$, both sides are
contractible because $X(d)$ is an $\unitad(d)$-module. Cellularizing
at any prime of dimension $>d$, both sides are contractible because
$X(d)$ is a torsion module. For primes of dimension $d$ note that we
are considering the map 
$$\unitad(\bd)\lra \prod_{\dim(\fp)=d}L_{\fp}e_{\fp}\unitad(\bd). $$
If we choose  a prime $\fq$ is of dimension $d$ and cellularize,
 factors for $\fp\neq \fq$ are trivial and we have
$$\unitad(\bd)\lra L_{\fq}e_{\fq}\unitad(\bd), $$
and this is an equivalence at $\fq$ as required. 
\end{proof}

\begin{remark}
 The fact that the adelic model is equivalent to the strict
 homotopy limit relies on special properties of the adelic fracture
 square (essentially the stratification by dimension). 
Taking a commutative ring to its category of chain complexes does not preserve arbitrary finite homotopy limits. For example, we
may, we may consider the diagram of rings
$$\xymatrix
{&\Z\dto \\
\Z \rto &\Q 
}$$
whose pullback is  $\Z$. The corresponding diagram of module
categories is not a strict homotopy limit. Indeed, it is clear that
the diagram
$$\xymatrix{
&\Z/p\dto \\
0 \rto &0
}$$
is in the strict limit but does not come from a $\Z$-module. 
\end{remark}

\section{Examples}\label{sec:examples}
We comment on three classes of examples, from  algebraic geometry,  algebraic topology and representation theory. 

\subsection{Algebraic geometry}

We have already seem the prototypical example coming from the theory
of commutative algebra, that is, we take $R$ a finite dimensional
 Noetherian ring and we can consider the projective model
structure on $\textbf{Ch}(R \textbf{-mod})$ whose homotopy category is
exactly the unbounded derived category of $R$. We also know that the
Balmer spectrum $\spcc(\sfD (R))$ is in  bijection with
$\spec(R)$. The adelic approximation theorem then tells us that there
is a natural way to reconstruct $R$ from localized completed rings. In
the case of a 2-dimensional irreducible Noetherian ring $R$, we would get a cube of the following form:
$$\xymatrixcolsep{0ex}\xymatrixrowsep{5ex}\xymatrix {
& \displaystyle{\prod_{\fp}}(R_{\fp}^{\wedge})_{\fp}\ar[rr] \ar[dd]|\hole&&R_{(0)}\otimes  \displaystyle{\prod_{\fp}}(R_{\fp}^{\wedge})_{\fp}\ar[dd]\\
R\ar[rr] \ar[dd] \urto&& R_{(0)} \ar[dd] \urto&\\
& \displaystyle{\prod_{\fp}}R_{\fp}\otimes  \displaystyle{\prod_{\fm\leq \fp}}R_{\fm}^{\wedge}\ar[rr]|>>>>>>>>>>>\hole &&R_{(0)}\otimes 
 \displaystyle{\prod_{\fp}}R_{\fp}\otimes  \displaystyle{\prod_{\fm\leq \fp}}R_{\fm}^{\wedge}\\
 \displaystyle{\prod_{\fm}}R_{\fm}^{\wedge} \ar[rr]  \urto&& R_{(0)} \otimes  \displaystyle{\prod_{\fm}}R_{\fm}^{\wedge} \urto&
}$$

The category of $R$-modules is equivalent to the category of 
quasi-coherent sheaves on $\operatorname{Spec}(R)$.  
Accordingly,  chain complexes of $R$-modules is a special case of a
wider class of examples arising from algebraic geometry.

More generally, we let $X$ be a topologically Noetherian scheme. In~\cite[Corollary
5.6]{BalmerSpc}, it was shown that 
$\spcc(\sfD (\textbf{Qcoh}(\mathcal{O}_X)))$ can be used to recover the
scheme $X$.  That is, there is a homeomorphism 
$f \colon X \xrightarrow{\sim} \spcc(\sfD (\textbf{Qcoh}(\mathcal{O}_X)))$ with
$$f(x) = \{a \in \textbf{Perf}(X) \mid a_x \simeq 0 \text{ in }
\textbf{Perf}(\cO_{X,x})\} \qquad \text{for all } x \in X.$$

Since $X$ is topologically Noetherian, $\sfD
(\textbf{Qcoh}(\mathcal{O}_X))$ is a Noetherian tensor-triangulated category which is generated by a
single perfect complex $P$~\cite[Tag 09IS]{stacks-project}. 

Our main result therefore shows that 
$\textbf{Ch}(\textbf{Qcoh}(X))$ with the projective model structure is 
Quillen equivalent to $\cLim \left( \unitad \textbf{-mod} \right) $, 
where the values of $\unitad$ are the adelic restricted products built from 
the localized completed stalks $((\cO_X)_x^{\wedge})_x$

This extends to algebraic stacks.  Quasi-coherent sheaves on suitable
algebraic stacks carry a compatible model structure~\cite{MR3384483}.
When $\cX$ is a tame stack, the Balmer spectrum of the derived
category of  perfect complexes of quasi-coherent $\cO_{\cX}$-modules has been
 computed~\cite{MR3529092}. 
Under the tameness condition, the category of quasi-coherent sheaves on $\cX$ is a compactly rigidly-generated model category. 

\subsection{Algebraic topology}
We describe two  examples. In chromatic homotopy theory at $p$ there
is only one prime of each dimension so the adelic flavour is
not expressed, but our approach gives a module-theoretic version of  chromatic
fracture methods.  The second example is rational equivariant
cohomology theories, and indeed,  the principal motivation this work
was to understand the underpinnings of the results of \cite{tnqcore} 

\subsubsection{Chromatic homotopy theory}
\newcommand{\SpO}{\textbf{Sp}^{\mathrm{O}}}
We pick a suitable model for the category
of spectra, such as orthogonal spectra~\cite{MMSS}
$\SpO$.  The homotopy category
is a rigidly small-generated tensor-triangulated category (with the
sphere spectrum $\bS^0$ as small generator). The Balmer
spectrum of $\textbf{Spectra}^\omega$ was
identified in~\cite{HopkinsSmith}. To describe this, 
for $n\geq 1 $ we write $K(p, n)$ for the $n$th mod 
$p$ Morava $K$-theory,  extending this to let $K(p, \infty)=H\Fp$ denote 
mod $p$ cohomology and $K(p, 0)=H\Q$ denote 
rational cohomology.  The primes are
$$\cP (p,n)=\{ X\st K(p,n)_*(X)=0\}. $$
When $n=0$, this is independent of $p$ and consists of the finite 
torsion spectra. 

For each prime integer prime $p>0$ the Balmer primes are linearly ordered: 
$$\cP (p, 0)\supset \cP (p, 1)\supset \cP (p, 2)\supset \cdots 
\supset \cP(p, \infty). $$

Corresponding to $\cP (p, n)$ we have a localization and 
a completion.  The localization 
$L_{\cP (p, n)}$ is finite Bousfield localization with
respect to $K(p, 0)\vee K(p, 1)\vee \ldots \vee K(p, n)$. 
If $p$ is understood this is usually written $L_n^f$. When the telescope conjecture
is true (it is known only for  $n=0, 1$) this coincides
with the ordinary Bousfield localization with respect to 
$K(p, 0)\vee K(p, 1)\vee \ldots \vee K(p, n)$. The completion corresponding to $\cP (p,n)$ is 
the Bousfield localization $L_{K(p, n)}$ with respect to
$K(p, n)$. 

From the infinite decreasing chain  we see that this Balmer
spectrum is not Noetherian, and the infinite primes $K(n, \infty)$ are not
visible.  For the  prime $\cP (p,n)$  with $n$ finite, we may take
$K_{\cP (p,n)}$ to be any generalized Smith-Toda complex $\unit/p^{i_0},
v_1^{i_1}, \ldots, v_n^{i_n}$.  In particular, all
primes $\cP_{p,n}$ with $n$ finite are visible. 

However, we can take the $L_n$ localization of $\SpO$ at a fixed prime $p$ which will truncate the Balmer spectrum to a single branch (namely the $p$-branch) truncated to level $n$.

Now, let us fix a prime $p$ consider the category $L_{n}\SpO$ 
of those spectra with chromatic level as most
$n$. Then from the above discussion we know that this is a Noetherian
model category where the monoidal unit coincides with small generator
$L_n \bS^0$. Our main result therefore shows that $L_{n}\SpO$ is Quillen equivalent to $\cLim \left( \unitad \textbf{-mod} \right) $, 
where the values of $\unitad$ are the localizations $L_tL_{K(s)}\bS^0$
for $n\geq t\geq s$. 

\begin{remark}
In chromatic homotopy theory,  results similar to ours are well
known. To start with, the one
dimensional case for assembling two adjacent chromatic levels is the chromatic fracture square and \cite{cubical} extends this
to higher cubes. However these differ from our results in two ways. 

Firstly, they express the $E(n)$-local
category in terms of the  $L_{\cP (p,i)}$- and $\Lambda_{\cP
  (p,i)}$-localizations  of the category for $i\leq n$ whereas our
model is in  terms of {\em module} categories
over the localized rings. Secondly, for dimensions higher than 2, the
diagrams of \cite{cubical} are more complicated, and oriented limits
are used.  
\end{remark}

\subsubsection{Rational torus-equivariant cohomology theories}
We finally return to the examples that spawned this research,
namely the algebraic models of $G$-equivariant rational cohomology
theories. The results in the present  paper reproduce the
first step  in the construction of the algebraic models for tori. The final
result requires (i) an argument to show that we may take fixed points
without losing information (ii) application of Shipley's equivalence 
(iii)  proving commutativity and (iv) a formality argument~\cite{tnqcore,BGKS,MR3633720,MR3704254}.

For a general compact Lie group $G$, where we recall that an
inclusion $K \subseteq H$ of subgroups is said to be \emph{cotoral} if
$K$ is normal and $H/K$ is a torus. As described in Subsection
\ref{runningex2}, as a poset, the  Balmer spectrum  consists of the
 conjugacy classes of subgroups under cotoral inclusion.

\subsubsection{$G$ a finite group}
The easiest case is when $G$ is finite, when we recover the Quillen
equivalence of \cite{Barnese}. The Balmer spectrum of
$\textbf{Sp}_\bQ^G$ is a discrete space, with the points being
conjugacy classes of subgroups. In particular, the Balmer spectrum is
0-dimensional. Module categories over a finite product of rings splits
correspondingly ~\cite{Barnese}. 

\begin{cor}
Let $G$ be a finite group, then 
$$\textbf{Sp}_\bQ^G \simeq_Q \prod_{H \in \operatorname{Sub}(G)/G} ( L_{H} \bS ) \textbf{-mod} \simeq_Q \prod_{H \in \operatorname{Sub}(G)/G} \textbf{Ch}(\bQ[W_GH] \textbf{-mod}).$$
\end{cor}

\begin{proof}
In the 0-dimensional case Theorem \ref{mame} states that $\cC$ is
equivalent to the category of modules over the product ring
$\prod_HL_H \bS$. Barnes's Theorem shows this is a product of module
categories of the individual rings $L_H\bS$: this gives the first equivalence. Morita theory and the
formality of $[L_H\bS, L_H\bS]^G=\Q W_G(H)$ gives the second
equivalence. 
\end{proof}

\subsubsection{$G=\T$, the circle group} Here we recover a theorem of
\cite{tnqcore}. The Balmer spectrum of $\textbf{Sp}_\bQ^{\T}$ is
$$
\spcc(\textbf{Sp}^{\T}_\bQ) = 
\begin{gathered}
\xymatrixcolsep{0ex}\xymatrixrowsep{0ex}\xymatrix{&& &\T &\\
 &&&&\\
 &&&&\\
\cdots &C_4 \ar[uuurr]&C_3\ar[uuur] &C_2 \ar[uuu]&C_1\ar[uuul]
}
\end{gathered}
$$
where $C_i$ is the cyclic group of order $i$. The diagram $\unitad$ is then
$$\xymatrix{
& L_T \bS \ar[d] \\ \prod_n \Lambda_{C_n} \bS \ar[r] & L_T \prod_n \Lambda_{C_n} \bS \rlap{ .}
}$$
which can be written as
$$\unitad=\left(\begin{gathered}\xymatrix{& \widetilde{E}\cF \ar[d] \\
      \prod_n DE \langle n \rangle_+ \ar[r] & \widetilde{E}\cF \wedge
      \prod_n DE \langle n \rangle_+ } \end{gathered}\right) \simeq 
\left(\begin{gathered}\xymatrix{& \widetilde{E}\cF \ar[d] \\ DE\cF_+ \ar[r] & \widetilde{E}\cF \wedge DE\cF_+ } \end{gathered}\right)$$
which is the usual Tate square for rational $\T$-equivariant
homotopy theory \cite{GMTate}.

\begin{cor}
There is a Quillen equivalence
$$\textbf{Sp}_\bQ^{\T} \simeq \cLim
\left(\begin{gathered}\xymatrix{& \widetilde{E}\cF \textbf{-mod} \ar[d] \\ DE\cF_+ \textbf{-mod} \ar[r] & \widetilde{E}\cF \wedge DE\cF_+ \textbf{-mod}} \end{gathered}\right).$$
\end{cor}

\subsubsection{Aside on invertible objects}
Continuing with rational $\T$-spectra, the adelic  point of view gives an interesting perspective on spheres
$S^V$, where $V^{\T}=0$. Spheres are invertible objects, and as
such their models are invertible at
each Balmer prime.  Because the rings $\Q$ and $\Q[c]$ are local,
invertible modules are free so the module at each prime is an integer
suspensions of the ring.  More explicitly,  the model  of $S^V$ is 
$$\xymatrix{
& \Q\dto \\
\prod_F \Sigma^{|V^F|}\Q[c] \rto & \cE^{-1} \prod_F\Q[c].
}$$
All three terms are invertible modules over their respective rings and
maps between them given by units in the respective rings will give
isomorphisms. Because of grading, isomorphism is rather
straightforward in this case. 

More generally, we may have  a cospan
$$\xymatrix{
 & B\dto \\
C\rto & D 
}$$
of commutative rings from the adelic square of a
1-dimensional Balmer spectrum. We may ask for isomorphism classes of objects
which are free on one generator at each Balmer prime but where the horizontal and
vertical maps are non-standard. It may even happen that the
$B$-module, $C$-module and $D$-module are $B,C$ and $D$
respectively. In this case the identification of the images of $B$ and
$C$ is given by a unit in $D$, so that  elements of 
$$\cok \left( B^{\times}\times C^{\times}\lra D^{\times} \right)  $$
will give rise to exotic invertible objects. 

One type of  example arises when $H^1_{ad}=0$ (i.e., when the maps $B\lra D$ and $C\lra D$ are
jointly surjective). For example, by \cite{adelicc} this occurs when
we have the adelic square for the ring of integers $A$ in a number
field. In this case invertible elements correspond to elements of the
class group of $A$, and there is a Mayer-Vietoris sequence for
$K_0(A)$. A second class of examples arise for $H^1_{ad}\neq
0$.  This occurs  for rational $G$-equivariant cohomology
theories for a 1-dimensional abelian compact Lie group $G$  (as we have  
just seen for $G=\T$) where representation spheres give invertible
objects. It also arises from quasi-coherent sheaves on a projective
curve, where line bundles give invertible objects.

\subsubsection{$G$ a torus}

Let $\mathbb{T}^n$ be an $n$-dimensional torus, then the category of
$\mathbb{T}^n$-equivariant rational cohomology theories is an
$n$-dimensional Noetherian model category. For  $r=2$, \cite{tnqcore} constructs the Adelic
Approximation  Cube 
$$\xymatrixcolsep{-4ex}\xymatrixrowsep{8ex}\xymatrix{
& \prod_H S^{\infty V(H)} \wedge DE \cF/H_+ \ar[rr] \ar[dd]|\hole&& \ar[dd] S^{\infty V(G)} \wedge \prod_H \left( S^{\infty V(K_s)} \wedge DE\cF / {K_s}_+ \right) \\
S^0 \ar[rr] \ar[ur] \ar[dd]&&\ar[dd] S^{\infty V(G)} \ar[ur]  &\\
& \prod_H S^{\infty V(H)} \wedge DE \cF_+ \ar[rr]|<<<<<<<<<<<\hole &&
S^{\infty(V(G))} \wedge \prod_H \left( S^{\infty V(H)} \wedge DE\cF_+ \right) \\
DE \cF_+ \ar[rr] \ar[ur] && S^{\infty V(G)} \wedge DE \cF_+ \ar[ur] &
}$$
where $S^{\infty V(H)}=\bigcup_{V^H=0}S^V$ and $\cF/H$ is the family
of subgroups with finite image in $G/H$. 

We note that  for ranks $r\geq 2$ this cube from \cite{tnqcore} uses $DE\cF_+$
because of its multiplicative properties. This differs from
slightly different from the one used in this paper, because $DE\cF/H_+$ involves a product
over all finite subgroups instead of just those contained in the
subgroup of the next dimension. The two cubes have the same pullback
(see \cite{adelicc}).

\subsection{Representation theory}

Our final example comes from modular representation theory. 
 We take $G$ to be a finite group and $k$ a field with characteristic
 dividing the order of $G$, we then consider the category of
 $kG$-modules. Say that two morphisms $f,g \colon M \to N$ are {\em
   stably homotopic} if $f-g \colon M \to N$ factors through a
 projective  module. The weak equivalences are then the homotopy
 equivalences. The fibrations are the epimorphisms and the
 cofibrations are the monomorphisms. The homotopy category of this
 model category is the stable module category
 $\underline{\textbf{Mod}}-kG$ \cite[2.2.12]{HoveyMC}.

This is a rigidly small-generated tensor-triangulated category
with unit object the trivial representation $k$. 
The Balmer spectrum of the small objects is the same as the space
which underlies the projective scheme associated to the group
cohomology ring~\cite{BensonCarlsonRickard}:
$$\spcc (\underline{\textbf{Mod}}-k G) \cong \operatorname{Proj}
H^* (G;k).$$
In particular, since $H^*(G;k)$ is a Noetherian graded ring by
Venkov's Theorem, the Balmer spectrum is Noetherian. 

Accordingly, the adelic model for $\underline{\textbf{Mod}}-kG$
is closely related to the  the adelic model for quasi-coherent modules over
$\operatorname{Proj} H^* (G;k).$ In fact the fibre  over a
homogenous prime $\fp$ in the respective adelic models is $(C^*(BG;k)_{\fp}^{\wedge})_{\fp}$ in the
first case and  $(H^*(BG;k)_{\fp}^{\wedge})_{\fp}$ in the second case. 

Usually these are different: for example, even if $G$ is cyclic of odd
prime order, one sees there are Massey products showing the cochain
algebra is not formal.

\end{document}